\documentclass[bj,16pt]{imsart}
\setattribute{journal}{name}{}
\usepackage{latexsym,epsfig,amssymb,amsmath,amsthm,url,bbm,enumerate,float,dsfont}
\usepackage{verbatim}
\usepackage[usenames,dvipsnames]{color}
\usepackage{longtable}
\usepackage{graphicx,hyperref}
\RequirePackage[square, numbers, sort & compress]{natbib}
\allowdisplaybreaks 
\numberwithin{equation}{section}
\usepackage{cleveref}

\newtheorem{Theorem}{Theorem}[section]
\crefname{Theorem}{Theorem}{Theorems}
\Crefname{Theorem}{Theorem}{Theorems}
\newtheorem{Corollary}[Theorem]{Corollary}
\crefname{Corollary}{Corollary}{Corollaries}
\Crefname{Corollary}{Corollary}{Corollaries}
\newtheorem{Proposition}[Theorem]{Proposition}
\crefname{Proposition}{Proposition}{Propositions}
\Crefname{Proposition}{Proposition}{Propositions}
\newtheorem{Lemma}[Theorem]{Lemma}
\crefname{Lemma}{Lemma}{Lemmata}
\Crefname{Lemma}{Lemma}{Lemmata}

\crefname{Assumption}{Assumption}{Assumptions}

\Crefname{Assumption}{Assumption}{Assumptions}

\usepackage{chngcntr}
\counterwithin{figure}{section}

\newtheoremstyle{note}%
{3pt}%
{3pt}%
{}%
{}%
{\bf}%
{}%
{.5em}%
{\thmname{#1}\thmnumber{ #2} \thmnote{\sc{(#3)}}}

\theoremstyle{note}
\newtheorem{Definition}[Theorem]{Definition}
\crefname{Definition}{Definition}{Definitions}
\Crefname{Definition}{Definition}{Definitions}
\newtheorem{Remark}[Theorem]{Remark}
\crefname{Remark}{Remark}{Remarks}
\Crefname{Remark}{Remark}{Remarks}

\newtheorem{Example}[Theorem]{Example}
\crefname{Example}{Example}{Examples}
\Crefname{Example}{Example}{Examples}

\newtheorem{assumptionletter}{Assumption}

\newtheorem{modelletter}[assumptionletter]{Model}

\def\la{\leftarrow}

\def\cont{\stackrel{t \to \infty}{\to}}
\def\conv{\stackrel{v}{\to}}
\def\cinP{\stackrel{P}{\to}}

\def\weak{\Rightarrow}

\def\C{\mathbb{C}}

\def\E{\mathbf{E}}
\def\Esp{\mathbb{E}}

\def\M{\mathbb{M}}

\def\P{{\rm{Pr}}}
\def\R{\mathbb{R}}

\def\N{\mathbb{N}}

\def\bV{\boldsymbol V}

\def\bY{\boldsymbol Y}
\def\bZ{\boldsymbol Z}

\def\dx{\mathrm d x}
\def\dy{\mathrm d y}

\def\RV{\mathcal{RV}}
\def\MRV{\mathcal{MRV}}
\def\HRV{\mathcal{HRV}}

\def\MES{\text{MES}}
\def\VaR{\text{VaR}}
\def\MME{\text{MME}}

\definecolor{darkred}{RGB}{139,0,0}
\definecolor{darkgreen}{RGB}{0,100,0}

\newcommand{\ov}{\overline}

\newcommand{\wh}{\widehat}

\newcommand{\1}{{\mathds{1}}}

\begin{document}


\begin{frontmatter}

\title{Risk contagion under regular variation and asymptotic tail independence}

\begin{aug}
  \author{\fnms{Bikramjit}  \snm{Das}\ead[label=e1]{bikram@sutd.edu.sg}}\thanksref{T1}
 \and
 \author{\fnms{Vicky} \snm{Fasen} \ead[label=e2]{vicky.fasen@kit.edu}}
 \thankstext{T1}{B. Das gratefully acknowledges support from MOE Tier 2 grant MOE-2013-T2-1-158. B. Das also acknowledges
  hospitality and support from Karlsruhe Institute of Technology during a visit in June 2015.}

  \runauthor{B. Das \and V. Fasen}

  \affiliation{SUTD\thanksmark{m1} \and Karlsruhe Institute of Technology\thanksmark{m2}}

  \address{Singapore University of Technology and Design\\8 Somapah Road, Singapore 487372 \\
           \printead{e1}}

  \address{Karlsruhe Institute of Technology\\ Englerstra{\ss}e 2, 76131 Karlsruhe\\
          \printead{e2}}

\end{aug}

\begin{abstract}
Risk contagion concerns  any entity dealing with large scale risks. Suppose   $\bZ=(Z_1, Z_2)$ denotes a risk vector pertaining to two components in some system.
A relevant measurement of risk contagion would be to quantify the amount of influence of high values of $Z_2$  on $Z_1$.
This can be measured in a variety of ways.
In this paper, we study  two such measures: the quantity $\E[(Z_1-t)_+|Z_2>t]$ called \emph{Marginal Mean Excess} (MME)  as well as  the related quantity $\E[Z_1|Z_2>t]$  called
\emph{Marginal Expected Shortfall} (MES). Both quantities are indicators of risk contagion and useful in various applications  ranging from finance,  insurance and systemic risk to environmental and climate risk.
We work under the assumptions of multivariate regular variation, hidden regular variation  and asymptotic tail independence for the risk vector $\bZ$.
Many broad and useful model classes satisfy these assumptions. We present several examples and derive the asymptotic behavior of both MME and MES as the threshold  $t\to\infty$.
We observe that although we assume asymptotic tail independence in the models, MME and MES  converge to $\infty$ under very general conditions; this reflects that the underlying
weak dependence in the model still remains significant.
Besides the  consistency of the empirical estimators  we introduce an extrapolation method based on extreme value theory
to estimate both  MME and MES for high thresholds $t$
where little data are available.
We show that these estimators are consistent and illustrate our methodology in both simulated and real data sets.
 \end{abstract}

\begin{keyword}[class=AMS]
\kwd[Primary ]{62G32}
\kwd{62-09}
\kwd{60G70}
\kwd[; secondary ]{62G10}
\kwd{62G15}
\kwd{60F05}
\end{keyword}

\begin{keyword}
\kwd{asymptotic tail independence}
\kwd{consistency}
\kwd{expected shortfall}
\kwd{heavy-tail}
\kwd{hidden regular variation}
\kwd{mean excess}
\kwd{multivariate regular variation}
\kwd{systemic risk}
\end{keyword}

\end{frontmatter}

\section{Introduction}\label{sec:intro}

  The presence of heavy-tail phenomena in data arising from a broad range of  applications spanning   hydrology \citep{anderson:meerschaert:1998}, finance
  \citep{smith:2003}, insurance
  \citep{embrechts:kluppelberg:mikosch:1997}, internet traffic
  \citep{crovella:bestavros:taqqu:1998,Resnick:2007}, social networks and random
  graphs \citep{durrett:2010, bollobas:borgs:chayes:riordan:2003}
  and  risk
management \citep{das:embrechts:fasen:2013,ibragimov:jaffee:walden:2011}
 is well-documented. Since heavy-tailed distributions often entail non-existence of some higher order moments, measuring and assessing dependence in jointly heavy-tailed random variables poses a few challenges.
Furthermore, one often encounters the phenomenon of \emph{asymptotic tail independence} in the upper tails; which means given two jointly distributed heavy-tailed random variables, joint occurrence of very high (positive) values is extremely unlikely.

  In this paper, we look at heavy-tailed random variables under the paradigm of \emph{multivariate regular variation} possessing asymptotic tail independence in the upper tails and we study the average behavior of one of the variables given that the other one is large in an asymptotic sense. The presence of asymptotic tail independence might intuitively indicate that high values of one variable will have little influence on the expected behavior of the other; we observe that such a behavior is not always true. In fact, under a quite general set of conditions, we are able to calculate the  asymptotic behavior of the expected value of a variable given that the other one is high.

 A major application of assessing such a behavior is in terms of computing systemic risk, where one wants to assess risk contagion among two risk factors in a system. Proper quantification of systemic risk has been  a topic of active research in the past few years; see  \cite{adrian:brunnermeier:2013, Biagini, Eisenberg, Feinstein,Brunnermeier:Cheridto:2014,Mainik:Schaaning} for further details. Our study concentrates on two such measures of risk in a bivariate set-up where both factors are heavy-tailed and possess asymptotic tail independence. Note that our notion of risk contagion refers to the effect of one risk on another and vice versa.   Risk contagion has other connotations which we do not address here; for example, it appears in causal models with time dependencies; see \cite{galiardini:gourieroux:2013} for a brief discussion. 

 First recall that for a random variable $X$ and $0<u<1$ the Value-at-Risk (VaR) at level $u$ is the quantile function $$\VaR_u(X):= \inf\{x\in\R: \P(X > x) \le 1- u \} = \inf\{x\in\R: \P(X \le x) \ge u \}.$$
Suppose  $\bZ=(Z_1, Z_2)$ denotes risk related to two different components of a system. 
We study the behavior of two related quantities which capture the
   expected behavior of one risk, given that the other risk is high.
        \begin{Definition}[Marginal Mean Excess] \label{definition:MME}
   For a random  vector $\bZ=(Z_1,Z_2)$ with $\E|Z_1|<\infty$ the \emph{Marginal Mean Excess} (MME) at level $p$ where $0<p<1$ is defined as:
  \begin{align}\label{def:MME}
  \MME(p) = \E\left[\left(Z_1-\VaR_{1-p}(Z_2)\right)_+ \big|Z_2>\VaR_{1-p}(Z_2)\right].
  \end{align}
   We interpret the MME as the expected excess of one risk $Z_{1}$ over the Value-at-Risk of $Z_2$ at level $(1-p)$
   given that the value of $Z_{2}$ is already greater than the same Value-at-Risk.
   \end{Definition}

   \begin{Definition}[Marginal Expected Shortfall] \label{definition:MES}
   For a random  vector $\bZ=(Z_1,Z_2)$  with $\E|Z_1|<\infty$ the \emph{Marginal Expected Shortfall} (MES) at level $p$ where $0<p<1$ is defined as:
  \begin{align}\label{def:MES}
  \MES(p) = \E\left[Z_1|Z_2>\VaR_{1-p}(Z_2)\right].
  \end{align}
   We interpret the MES as the expected shortfall of one risk given that the other risk is higher than its Value-at risk at level $(1-p)$. Note that  smaller values of $p$ lead to higher values of $\VaR_{{1-p}}$.
   \end{Definition}

 In the context of systemic risk, we may think of the conditioned variable $Z_2$ to be the risk of the entire system (for example, the entire market) and the variable $Z_1$ as one component of the risk (for example, one financial institution). Hence, we are interested in the average or expected behavior of one specific component when the entire system is in distress. Although the problem is set up  in a systemic risk context, the asymptotic behaviors of $\MME$ and $\MES$ are of interest in scenarios of risk contagion in a variety of disciplines.

Clearly, we are interested in computing both $\MME(p)$ and $\MES(p)$  for small values of $p$, which translates to $Z_2$ being over a high threshold $t$. In other words we
are interested in estimators of  $\E[(Z_1-t)_{+}| Z_2 >t]$ (for the MME) and  $\E[Z_1| Z_2 >t]$ (for the MES) for large values of $t$.
An estimator for $\MES(p)$ has been proposed by \cite{cai:einmahl:dehaan:zhou:2015} which is based on
the asymptotic behavior of $\MES(p)$; if $Z_1\sim F_1$ and $Z_2\sim F_2$, define
\begin{align}\label{def:R}
   R(x,y) := \lim_{t\to\infty}t\,\P\left(1-F_1(Z_1) \le \frac xt, 1-F_2(Z_2)\le \frac yt \right)
   \end{align}
for $(x,y) \in \left[0,\infty\right)^2$.
  It is shown in \cite{cai:einmahl:dehaan:zhou:2015}  that
\begin{eqnarray} \label{i:1}
    \lim_{p\to 0}\frac{1}{\VaR_{1-p}(Z_1)}\MES(p)=\int_0^{\infty} R(x^{-\alpha_1},1) \;\mathrm dx
\end{eqnarray}
 if $Z_1$ has a regularly varying tail with tail parameter $\alpha_1$. In \cite{Joe:Li} a similar result is presented under the further assumption of
  multivariate regular variation of the vector $\bZ=(Z_1,Z_{2})$; see \cite{Zhu:li:2012, Hua:Joe:2012}
  as well in this context.   Under the same assumptions, we can check that
 \begin{eqnarray} \label{i:2}
    \lim_{p\to 0}\frac{1}{\VaR_{1-p}(Z_1)}\MME(p)=\int_c^{\infty} R(x^{-\alpha_1},1) \;\mathrm dx
\end{eqnarray}
where $$c = \lim\limits_{p\to 0} \frac{\VaR_{1-p}(Z_{2})}{\VaR_{1-p}(Z_{1})}, $$ if $c$ exists and is finite. For $c$ to be finite we require that $Z_{1}$ and $Z_{2}$ are (right) tail equivalent ($c>0$) or $Z_{2}$ has a lighter (right) tail than $Z_{1}$ ($c=0$). Note that in both \eqref{i:1} and \eqref{i:2}, the rate of increase of the risk measure
is determined by the tail behavior of $Z_1$; the tail behavior of $Z_2$ has no apparent influence.
However, these results make sense only when the right hand sides of \eqref{i:1} and
\eqref{i:2} are both non-zero and finite. Thus, we obtain that  as $p\downarrow 0$,
\begin{eqnarray*}
    \MME(p)\sim\text{const. }\VaR_{1-p}(Z_1),
    \quad \text{ and } \quad
    \MES(p)\sim\text{const. }\VaR_{1-p}(Z_1).
\end{eqnarray*}
Unfortunately, if $Z_1, Z_2$ are asymptotically upper tail  \emph{{independent}} then $R(x,y)\equiv0$ (see \cref{Remark:3} below)
which implies that the limits in \eqref{i:1} and \eqref{i:2} are both $0$ as well and hence, are  not that useful.

Consequently, the results in \cite{cai:einmahl:dehaan:zhou:2015} make sense only if the random vector $\bZ$ has positive upper tail dependence,
 which means that, $Z_1$ and $Z_2$ take high values together with a positive probability; examples of multivariate regularly varying random vectors producing such strong
 dependence can be found in \cite{Hua:Joe:2014}.
 A classical example for asymptotic tail independence, especially in financial risk modeling,  is when  the risk factors $Z_1$ and $Z_2$ are both
 Pareto-tailed with a Gaussian copula and any correlation
$\rho<1$ \citep{das:embrechts:fasen:2013}; this model has asymptotic upper tail  independence leading to $R\equiv 0$. The results in
\eqref{i:1} and \eqref{i:2} respectively, and hence, in \cite{cai:einmahl:dehaan:zhou:2015} provide a null estimate which is not very informative. Hence, in such a case one might be inclined to believe that
 $\E(Z_1|Z_2>t)\sim\E(Z_1)$ and $\E((Z_1-t)_+|Z_2>t)\sim 0$ as $Z_1$ and $Z_2$
are asymptotically tail independent.
However, we will see that depending on the Gaussian copula parameter $\rho$
we might even have
$\lim_{t\to\infty}\E((Z_1-t)_+|Z_2>t)=\infty$. Hence, in this case it would be nice if we could find  the right rate of convergence of $\MME(p)$ to a non-zero constant.

In this paper we investigate the asymptotic behavior of $\MME(p)$ and $\MES(p)$ as $p\downarrow 0$ under the assumption of regular variation and hidden regular variation
of the risk vector $\bZ$ exhibiting asymptotic upper tail independence. We will see that for a very general class of models
$\MME(p)$ and $\MES(p)$, respectively behave like a regularly varying function with negative index for $p\downarrow 0$, and hence, converge to $\infty$
although the tails are asymptotically tail independent. However, the rate of convergence is slower than in the asymptotically tail dependent
case as presented in  \cite{cai:einmahl:dehaan:zhou:2015}. This result is an  interplay between the tail behavior and the strength of dependence of the two variables in the tails.
The behavior of MES in the asymptotically  tail independent case has been addressed to some extent in
\cite[Section 3.4]{Hua:Joe:2014} for certain copula structures with Pareto margins.  We address the asymptotically tail independent case in further generality. For the MME, we can provide results with fewer technical assumptions
than for the case of MES and hence, we cover a broader class of asymptotically tail independent
models. The knowledge of the asymptotic behavior of the MME and the MES helps us in proving
consistency of their empirical estimators. However, in a situation where  data are scarce or even unavailable in the tail region of interest,
an empirical estimator is clearly unsuitable. Hence, we also provide consistent estimators using methods from extreme value theory which work when data availablility is limited in the tail regions.

The paper is structured as follows: In Section \ref{sec:prelim} we briefly discuss the notion of multivariate and hidden regular variation. We also list a set of assumptions that we impose on our models in order to obtain limits of the quantities MME and MES under appropriate scaling.  The main results of the paper regarding the asymptotic behavior of the MME and the MES
 are discussed in \Cref{sec:asymptotic}.  In  Section \ref{subsec:Models}, we illustrate a few  examples which satisfy the assumptions under which we can compute asymtptoic limits of MME and MES; these include  additive models, the Bernoulli mixture model for generating hidden regular variation and a few copula models. Estimation methods for the risk measures MME and MES  are provided
  in Section \ref{sec:estimation}.
Consistency of the empirical estimators are the topic of \Cref{subsec:empirical}, whereas, we present
consistent estimators based on methods from extreme value theory in \Cref{subsec:EVT estimators}. 
Finally, we validate our method on real and simulated data in Section \ref{sec:simulation} with brief concluding remarks in \Cref{sec:conclusion}.

In the following we denote by $\conv$ vague convergence of measures, by
    $\weak$  weak convergence of measures and by
    $\cinP$  convergence in probability. For $x\in \R$, we write $x_{+} = \max(0,x)$.



\section{Preliminaries} \label{sec:prelim}
 For this paper we restrict our attention to non-negative random variables in a bivariate setting. We discuss multivariate and hidden regular variation in Section \ref{subsec:regvar}. 
A few technical assumptions that  we use throughout the paper are listed in Section \ref{subsec:assumptions}.  A selection of model examples that  satisfy  our assumptions is relegated to \Cref{subsec:Models}.


\subsection{Regular variation} \label{subsec:regvar}


First, recall that a  measurable function $f:(0,\infty)\to(0,\infty)$ is
 \textit{regularly varying} at $\infty$ with index $\rho \in \R$ if $$\lim_{t\to\infty} \frac{f(tx)}{f(t)}=x^{\rho}$$ for any $x>0$
and we write $f\in\RV_\rho$. If the index of regular variation is $0$ we call the function
slowly varying as well.
Note that in contrast, we say $f$ is regularly varying at
$0$ with index $\rho$ if $\lim_{t\to 0}f(tx)/f(t)=x^{\rho}$ for any $x>0$.  In this paper, unless otherwise specified, regular variation means regular variation at infinity. A random variable $X$ with distribution function $F$ has a regularly varying tail if $\overline{F}=1-F \in \RV_{-\alpha}$ for some $\alpha\ge 0$. We often write  $X\in \RV_{-\alpha}$ by abuse of notation.

 We use the notion of $\M$-convergence  to define regular variation in more than one dimension; for further details see \cite{lindskog:resnick:roy:2014,hult:lindskog:2006a,das:mitra:resnick:2013}. We restrict  to two dimensions here since we deal with bivariate distributions in this paper, although the definitions provided hold in general for any finite dimension.
Suppose $\C_0 \subset \C \subset \left[0,\infty\right)^2$ where
$\C_0$ and $\C$ are closed cones containing $\{(0,0)\} \in\R^2$.
By  $\M(\C\setminus\C_0)$ we denote the class of Borel measures on $\C\setminus\C_0$ which are
finite on subsets bounded away from  $\mathbb{C}_0$.
Then $\mu_n \stackrel{\M}{\to}\mu$ in $\M(\C \setminus \C_0)$ if $\mu_n(f)\to\mu(f)$ for all continuous
and bounded functions on $\C \setminus \C_0$ whose supports are bounded away from $\C_0$.

\begin{Definition}[Multivariate regular variation]\label{def:mrv}
 A random vector ${\bZ } =(Z_1,Z_2) \in \C$ is \emph{(multivariate) regularly
varying} on $\C \setminus \C_0$,
 if there {exist} a function $b(t) \uparrow \infty$ and a non-zero measure {$\nu(\cdot)\in \M(\C \setminus \C_0)$} such that as $t \to \infty$,
 \begin{equation}\label{eqn:ccc}
 \nu_{t}(\cdot):=t\,\P({{\bZ }}/{b(t)} \in \; \cdot \; ) \stackrel{\M}{\to} \nu(\cdot) \hskip 0.5 cm \text{in {$\M(\C \setminus \C_0)$.}}
\end{equation}
Moreover, we can check that the limit measure has the homogeneity property: $\nu(cA)=c^{-\alpha}\nu(A)$ for some $\alpha>0$. We write $\bZ\in \MRV(\alpha, b,\nu, \C \setminus \C_0)$ and sometimes write MRV for multivariate regular variation.
\end{Definition}

In the first stage, multivariate regular variation is defined on the space $\Esp=\left[0,\infty\right)^2\setminus\{(0,0)\} = \C \setminus \C_0$ where $\C=\left[0,\infty\right)^2$ and $\C_0=\{(0,0)\}$.
But sometimes we need to define further regular variation on subspaces of $\Esp$, since the limit measure $\nu$ as obtained in
\eqref{eqn:ccc} turns out to be concentrated on a subspace of $\Esp$. The most likely way this happens is through   \emph{asymptotic tail independence} of random variables.

\begin{Definition}[asymptotic tail independence]\label{def:asyind}
A random vector $\bZ =(Z_1, Z_2) \in \left[0,\infty\right)^2$ is called \linebreak \emph{asymptotically independent (in the upper tail)}  if
\begin{eqnarray*}
    \lim_{p\downarrow 0}\P(Z_2>F_{2}^{\leftarrow}(1-p)|Z_1>F_{1}^{\leftarrow}(1-p))=0,
\end{eqnarray*}
where $Z_{i}\sim F_{i}, i=1,2$.
\end{Definition}
Asymptotic upper tail independence can be interpreted in terms of the survival copula of $\bZ$ as well.  Assume (w.l.o.g.) that $F_{1}$, $F_{2}$ are strictly increasing continuous distribution functions with
unique survival copula $\wh C$ (see \cite{Nelsen}) such that
\begin{eqnarray*}
    \P(Z_1>x,Z_2>y)=\wh C(\ov F_{1}(x),\ov F_{2}(y)) \quad  \text{for }(x,y)\in\R^2.
\end{eqnarray*}
Hence, in terms of the survival copula, asymptotic upper tail independence of $\bZ$ implies
\begin{align}\label{eq:AI_surv}
\lim_{p\downarrow 0}\frac{\wh C(p,p)}{p} & =  \lim_{p\downarrow 0} \frac{\Pr(Z_{1}>\ov F_{1}^{\leftarrow}(p),Z_{2}>\ov F_{2}^{\leftarrow}(p))}{\Pr(Z_{1}>\ov F_{1}^{\leftarrow}(p))} =  \lim_{p\downarrow 0} \Pr(Z_{2}> F_{2}^{\leftarrow}(p))|Z_{1}> F_{1}^{\leftarrow}(1-p)) =0.
\end{align}

Independent random vectors are trivially asymptotically tail independent.
 Note that  asymptotic upper tail independence of $\bZ\in \MRV(\alpha, b,\nu, \Esp)$ implies
 $\nu((0,\infty)\times(0,\infty))=0$ for the limit measure $\nu$.
 On the other hand, for the converse,  if $Z_1$ and $Z_2$ are both marginally regularly varying in the right tail with $\lim_{t\to\infty} \P(Z_1>t)/\P(Z_2>t) =1$, then $\nu((0,\infty)\times(0,\infty))=0$
 implies asymptotic upper tail independence as well (see \cite[Proposition 5.27]{Resnick:1987}). However, this implication does not hold true in general, e.g.,
for a regularly varying random variable $X \in \RV_{-\alpha}$ the random vector $(X,X^2)$ is multivariate regularly varying
with limit measure  $\nu((0,\infty)\times(0,\infty))=0$; but of course $(X,X^2)$ is asymptotically tail-dependent.

\begin{Remark} \label{Remark:3}
Asymptotic upper tail independence of $(Z_1,Z_2)$ implies that   
\begin{align*}
    R(x,y) &=\lim_{t\to\infty}t\, \P \left(1- F_1(Z_1) \le x/t, 1-F_2(Z_2) \le y/t \right)\\
        & =\lim_{t\to\infty}t\, \wh C\left(\frac{x}{t},\frac{y}{t}\right)
        \leq \max(x,y)\lim_{s\to 0}\frac{\wh C(s,s)}{s}=0 \quad (\text{using} \;\; \eqref{eq:AI_surv}).
\end{align*}
Hence,  the estimator presented in \cite{cai:einmahl:dehaan:zhou:2015} for MES
provides a trivial estimator in this setting.
\end{Remark}

  Consequently, in the asymptotically tail independent case where the tails are equivalent
 we would approximate the joint tail probability by
 $
    \P(Z_2>x|Z_1>x)\approx 0
 $
 for large thresholds $x$ and conclude that risk contagion between $Z_1$ and $Z_2$ is absent. This conclusion may be naive;
 hence the notion of \emph{hidden regular variation}  on $\Esp_0= \left[0,\infty\right)^2\setminus (\{0\}\times\left[0,\infty\right) \cup \left[0,\infty\right)\times\{0\}) =(0,\infty)^2$  was introduced in \cite{resnick:2002}.
{Note that we do not  assume that the marginal tails of $\bZ$ are necessarily equivalent in order to define hidden regular variation, which is usually done in \cite{resnick:2002}.
\begin{Definition}[Hidden regular variation]\label{def:hrv}
A regularly varying random vector $\bZ$ on $\Esp$ possesses \emph{hidden regular variation} on $\Esp_0 = (0,\infty)^2$ with index $\alpha_0\, (\ge \alpha >0)$ if there exist
 scaling functions $b(t)\in \RV_{1/\alpha}$ and $b_0(t)\in \RV_{1/\alpha_0}$
with $b(t)/b_0(t)\to\infty$ and limit measures $\nu,\nu_0$ such that
\begin{eqnarray*}
    \bZ \in \MRV({\alpha}, b, \nu, \Esp)\cap \MRV({\alpha_0}, b_0, \nu_0, \Esp_0).
\end{eqnarray*}
We write $\bZ\in \HRV({\alpha_0}, b_0, \nu_0)$ and sometimes write HRV for hidden regular variation.
\end{Definition}

For example, say $Z_1, Z_2$ are iid random variables with distribution function $F(x)=1-x^{-1}, x>1$.  Here $\bZ=(Z_1,Z_2)$ possesses MRV on $\Esp$, asymptotic tail independence and HRV on $\Esp_0$. Specifically,  $\bZ\in \MRV(\alpha=1, b(t)=t, \nu, \Esp) \cap \MRV(\alpha_{0}=2, b_{0}(t)=\sqrt{t}, \nu_0, \Esp_{0})$ where for $x>0,y>0$, $$\nu(([0,x]\times[0,y])^c) = \frac{1}{x} + \frac 1y \quad \text{ and } \quad \nu_0(\left[x,\infty\right)\times\left[y,\infty\right)) = \frac{1}{xy}.$$

\begin{Lemma}
    $\bZ \in \MRV({\alpha}, b, \nu, \Esp)\cap \HRV({\alpha_0}, b_0, \nu_0,\Esp_{0})$ implies that $\bZ$ is asymptotically
    tail independent.
\end{Lemma}
\begin{proof}
Let $b_{i}(t) = (1/(1-F_{i}))^{\la}(t)$ where $Z_{i}\sim F_{i}, i=1,2$. Due to the assumptions we have
\begin{eqnarray*}
    \lim_{t\to\infty}\frac{\max(b_1(t),b_2(t))}{b_0(t)}=\infty \quad \mbox{ and }\quad \liminf_{t\to\infty}\frac{\min(b_1(t),b_2(t))}{b_0(t)}\geq 1.
\end{eqnarray*}
Without loss of generality $b_1(t)/b_0(t)\to\infty$. Then for any $M>0$ there exists a $t_0=t_0(M)$ so that $b_1(t)\geq Mb_0(t)$ for any $t\geq t_0$.
Hence, for $x,y>0$
\begin{eqnarray*}
    \lim_{t\to\infty}t\, \P \left(1- F_1(Z_1) \le x/t, 1-F_2(Z_2) \le y/t \right)
        &=&\lim_{t\to\infty}t\, \P \left(Z_1 \geq b_1(t/x), Z_2 \geq b_2(t/y) \right)\\
        &\leq& \lim_{t\to\infty}t\, \P \left(Z_1 \geq Mb_0(t/x), Z_2 \geq 2^{-1}b_0(t/y) \right)\\
        &\leq &C_{x,y}\nu_0(\left[M,\infty\right)\times\left[2^{-1},\infty\right))\stackrel{M\to\infty}{\to}0,
\end{eqnarray*}
so that $\bZ$ is asymptotically tail independent (here $C_{x,y}$ is some fixed constant).
\end{proof}

\begin{Remark}
The assumption  $\bZ \in \MRV({\alpha}, b, \nu, \Esp)\cap\MRV({\alpha_0}, b_0, \nu_0, \Esp_0)$ and $\bZ$ is asymptotic upper tail independent already implies
that $\bZ\in \HRV({\alpha_0}, b_0, \nu_0)$; see \citep{resnick:2002,maulik:resnick:2005}. Consequently $\lim_{t\to\infty}b(t)/b_0(t)=\infty$ as well.
\end{Remark}


\subsection{Assumptions} \label{subsec:assumptions}
In this section we list assumptions on the random variables for which we show consistency of relevant estimators in the paper. Parts of the assumptions are to fix notations for future results.
\begin{assumptionletter} \label{cond:basic} $\mbox{}$ 
\begin{enumerate}
\item[(A1)]  Let $\bZ =(Z_1, Z_2) \in \left[0,\infty\right)^2$  such that $\bZ \in \MRV({\alpha}, b, \nu, \Esp)$  where
$$b(t)=\left(1/\P(\max(Z_1,Z_2)>\,\cdot\,)\right)^{\la}(t)=\ov F_{\max(Z_1,Z_2)}^{\leftarrow}(1/t)\in\RV_{1/\alpha}.$$ 
\item[(A2)] $\E|Z_1|<\infty$.
\item[(A3)] $b_2(t):=\ov F_{Z_2}^{\la}(1/t)$ for $t\geq 0$.
\item[(A4)]  Without loss of generality we assume that the support of $Z_1$ is $\left[1,\infty\right)$.  A constant shift would not affect the tail properties of MME or MES.
\item[(A5)]   $\bZ \in \MRV({\alpha_0}, b_0, \nu_0, \Esp_0)$ with $\alpha_0\ge \alpha\geq 1$, where
 $$b_0(t)=\left(1/\P(\min(Z_1,Z_2)>\,\cdot\,)\right)^{\la}(t)=\ov F_{\min(Z_1,Z_2)}^{\leftarrow}(1/t)\in\RV_{1/{\alpha_0}},$$
 and $b(t)/b_0(t)\to\infty$.
\end{enumerate}
\end{assumptionletter}

\begin{Lemma}
Let $\ov F_{Z_2} \in \RV_{-\beta}$, $\beta >0$. Then Assumption \ref{cond:basic} implies $\alpha\leq\beta\leq\alpha_0$.
\end{Lemma}
\begin{proof}
First of all, $\beta \ge \alpha$ since otherwise $\bZ \in \MRV({\alpha}, b, \nu, \Esp)$ cannot hold.
Moreover,
\begin{eqnarray} \label{a1}
    1\sim t\,\P(Z_1>b_0(t),Z_2>b_0(t))\leq t\, \P(Z_2>b_0(t))\in\RV_{1-\frac{\beta}{\alpha_0}}.
\end{eqnarray}
Thus, if $\alpha_0<\beta$ then $\lim_{t\to\infty}t \P(Z_2>b_0(t))=0$ which is a contradiction to \eqref{a1}.
\end{proof}
\begin{Remark}
In general, we see from this that under  Assumption~\ref{cond:basic},  $\liminf_{t\to\infty}t \P(Z_2>b_0(t))\geq 1$
and hence, for any $\epsilon>0$ there exist $C_1(\epsilon)>0$, $C_2(\epsilon)>0$  and $x_0(\epsilon)>0$ such that
 $$C_1(\epsilon) x^{-\alpha_0-\epsilon} \leq \P(Z_2>x)\leq C_2(\epsilon) x^{-\alpha+\epsilon}$$
for any $x\geq x_0(\epsilon)$.
\end{Remark}

We need a couple of more conditions, especially on the joint tail behavior of $\bZ = (Z_1,Z_2)$  in order to talk about the limit behavior of $\MME(p)$ and  $\text{MES}(p)$ as $p\downarrow 0$. We impose the following assumptions on the distribution of $\bZ$. Assumption (B1) is imposed to find the limit of MME  in \eqref{def:MME} whereas both (B1) and (B2) (which are clubbed together as  Assumption \ref{cond:dct1}) are imposed to find the limit in \eqref{def:MES}, of course, both under appropriate scaling.

\begin{assumptionletter} \label{cond:dct1} $\mbox{}$
\begin{enumerate}[(B1)]
    \item \label{B1}  ${\displaystyle \lim_{M\to\infty}\lim_{t\to\infty}\int_{M}^\infty \frac{\P(Z_1>xt, Z_2>t)}{\P(Z_1>t, Z_2>t)}\, \mathrm dx=0}$.
    \item \label{B2} ${\displaystyle \lim_{M\to\infty}\lim_{t\to\infty}\int_0^{1/M} \frac{\P(Z_1>xt, Z_2>t)}{\P(Z_1>t, Z_2>t)}\, \mathrm dx=0}$.
\end{enumerate}
\end{assumptionletter}
Assumption (B1) and Assumption (B2) deal with tail integrability near infinity and near zero for a specific integrand, respectively that comes up in calculating limits of
MME and MES.  The following lemma  trivially provides a
{sufficient} condition for (B1).

\begin{Lemma}
If there exists an integrable function $g:\left[0,\infty\right)\to \left[0,\infty\right)$
    with $$\sup_{t\geq t_0}\frac{\P(Z_1>y, Z_2>t)}{t\,\P(Z_1>t, Z_2>t)}\leq g(y)$$ for $y>0$ and some $t_0>0$
then (B1) is satisfied.
\end{Lemma}

\begin{Lemma} \label{Lemma:forAssB}
Let Assumption \ref{cond:basic} hold.
\begin{enumerate}[(a)]
    \item Then (B2) implies
    \begin{eqnarray} \label{B2*}
         \lim_{t\to\infty}\frac{\P(Z_2>t)}{t\,\P(Z_1>t, Z_2>t)}=0.
    \end{eqnarray}
    \item Suppose $\ov F_{Z_2}\in\RV_{-\beta}$ with $\alpha\leq\beta\leq\alpha_0$. Then
    $\alpha_0\leq \beta+ 1$ is a necessary and $\alpha_0< \beta+ 1$ is a sufficient condition for \eqref{B2*} to hold.
    Hence, $\alpha_0\leq \beta+ 1$ is a  necessary condition for Assumption  (B2) as well. \\
\end{enumerate}
\end{Lemma}
\begin{proof} $\mbox{}$\\
(a) \, Since the support of $Z_1$ is $\left[1,\infty\right)$ we get for large $t \ge M$, by (B2),
\begin{eqnarray*}
 \frac{\P(Z_2>t)}{t\P(Z_1>t, Z_2>t)}=\int_{0}^{1/t}\frac{\P(Z_1>xt, Z_2>t)}{\P(Z_1>t, Z_2>t)}\, \mathrm dx\leq
 \int_{0}^{1/M}\frac{\P(Z_1>xt, Z_2>t)}{\P(Z_1>t, Z_2>t)}\, \mathrm dx\stackrel{t,M\to\infty}{\to}0.
\end{eqnarray*}
But the left hand side is independent of $M$ so that the claim follows.\\
(b) In this case $$\frac{\P(Z_2>t)}{t\,\P(Z_1>t, Z_2>t)}\in\RV_{-\beta-1+\alpha_0}$$ from which the statement follows.
\end{proof}
\begin{Remark}
If $Z_1, Z_2$ are independent  then under the assumptions of {Lemma \ref{Lemma:forAssB}(b)}, $\alpha_0=\alpha+\beta$. Moreover if
 $1<\alpha\leq\beta$ then clearly $\alpha_0=\alpha+\beta>1+\beta$ and $\alpha_{0} \le 1+\beta$ cannot hold.
Hence, Assumption (B2) is not valid if $Z_1$ and $Z_2$ are independent.
In other words, Assumption (B2) signifies that although $Z_1,Z_2$ are asymptotically upper tail independent,
there is an underlying dependence between $Z_1$ and $Z_2$ which is absent in the independent case.
\end{Remark}


\section{Asymptotic behavior of the MME and the MES} \label{sec:asymptotic}

\subsection{Asymptotic behavior of the MME} \label{subsec:asymptotic:MME}

For asymptotically independent risks, from  \eqref{i:2} and \cref{Remark:3} we have that
 \begin{eqnarray*}
    \lim_{p\to 0}\frac{1}{\VaR_{1-p}(Z_1)}\MME(p)=0,
\end{eqnarray*}
which doesn't provide us much in the way of identifying the rate  of increase (or decrease) of $\MME(p)$.
The aim of this section is to get a version of \eqref{i:2} for the asymptotically tail independent case
 which is presented in the next theorem.

\begin{Theorem}\label{thm:MME_AI}
Suppose $\bZ=(Z_1,Z_2) \in [0,\infty)^{2}$  satisfies Assumption \ref{cond:basic}  and (B1). Then
\begin{align}\label{conv:main:2}
\lim\limits_{p\downarrow 0} \frac{pb_0^{\la}(b_2(1/p))}{b_2(1/p)} \MME(p)  =  \lim\limits_{p\downarrow 0} \frac{pb_0^{\la}(\VaR_{1-p}(Z_2))}{\VaR_{1-p}(Z_2)} \MME(p)=\int_1^{\infty} \nu_0((x,\infty)\times(1,\infty))\;\mathrm dx.
\end{align}
Moreover, $0<\int_1^{\infty} \nu_0((x,\infty)\times(1,\infty))\;\mathrm dx<\infty$.
\end{Theorem}

\begin{proof}
We know that for a non-negative random variable $W$, we have $\E W = \int_0^{\infty} \P(W>x)\, \mathrm dx.$ Let $t=b_2(1/p)$. Also note that $b_0^{\la}(t)=1/\P(\min(Z_1,Z_2)>t) =1/\P(Z_1>t,Z_2>t)$. Then
\begin{align}
\frac{pb_0^{\la}(b_2(1/p))}{b_2(1/p)} \MME(p)  & = \frac{\overline{F}_{Z_2}(t)b_0^{\la}(t)}{t} \;\E((Z_1-t)_+|Z_2>t)\nonumber \\
  								& = \frac{\P(Z_2>t)b_0^{\la}(t)}{t} \int_t^{\infty}\frac{\P(Z_1>x, Z_2>t)}{\P(Z_2>t)} \;\mathrm dx\nonumber \\
 								& =  \int_t^{\infty} \frac{\P(Z_1>x, Z_2>t)}{t\P(Z_1>t, Z_2>t)} \;\mathrm dx \nonumber \\
								& =  \int_1^{\infty} \frac{\P(Z_1>tx, Z_2>t)}{\P(Z_1>t, Z_2>t)} \;\mathrm dx	 =: \int_{1}^{\infty} \nu_{t}(x)\; \mathrm dx.	\label{def:nut}				
							\end{align}							
Observe that for $x\ge1$, by Assumption (A5),
\begin{align*}
  \nu_{t}(x) =\frac{\P(Z_1>tx, Z_2>t)}{\P(Z_1>t, Z_2>t)}  = b_0^{\la}(t) \P\left(\frac{\bZ}{t} \in (x,\infty)\times (1,\infty)\right) \cont \nu_0((x,\infty)\times (1,\infty)). 
  \end{align*}
We also have
\begin{align*}
 \nu_{t}(x) = \frac{\P(Z_{1}>tx,Z_{2}>t)}{\P(Z_{1}>t,Z_{2}>t)} \le 1, \quad x\ge1.
 \end{align*}
 Now, for $x\ge 1$, we have $ \nu_0((x,\infty)\times (1,\infty))\le  \nu_0((1,\infty)\times (1,\infty))= \lim_{t\to \infty} \nu_{t}(1) =1$. Hence, for any $M\ge 1$ we have $\int_{1}^M \nu_0((x,\infty)\times(1,\infty))\;\mathrm dx\leq M$. Therefore  using Lebesgue's Dominated Convergence Theorem,
\begin{align}\label{DCT}
 \lim_{t\to\infty}  \int_{1}^{M} \nu_{t}(x)\; \mathrm dx  & =  \int_{1}^{M} \nu_{0}((x,\infty)\times(1,\infty))\; \mathrm dx.
\end{align}
Next we  check that $0<\int_{1}^{\infty}\nu_0((x,\infty)\times (1,\infty)) \;\mathrm dx <\infty.$ 
Define for $M\ge1$,
\[ \psi_{M} := \lim_{t\to\infty}\int_{M}^{\infty} \nu_{t}(x) \mathrm dx.\]
By Assumption (B1), we have
 \begin{align}\label{psiM}\lim_{M\to\infty} \psi_{M} =0.\end{align} 
 Hence, there exists $M_{0}>0$ such that $|\psi_{M}|\le 1$ for all $M>M_{0}$.  Applying Fatou's Lemma, we know that for any $M>M_{0}$,
\begin{align*}
\int_{M}^{\infty}\nu_0((x,\infty)\times (1,\infty)) \;\mathrm dx & \le \liminf_{t\to \infty} \int_{M}^{\infty} \nu_{t}(x) \mathrm dx \le \psi_{M} \le 1.\end{align*}
Therefore, for fixed $M>M_{0}$,
\begin{align*}
\int_{1}^{\infty}\nu_0((x,\infty)\times (1,\infty)) \;\mathrm dx & = \int_1^M \nu_0((x,\infty)\times (1,\infty)) \,\mathrm d x + \int_M^{\infty} \nu_0((x,\infty)\times (1,\infty)) \,\mathrm d x  \le M +1 <\infty. 
\end{align*}
Moreover,
$\nu_0((x,\infty)\times(x,\infty))$ is homogeneous of order $-\alpha_0$ so that
\begin{eqnarray*}
    \int_{1}^\infty \nu_0((x,\infty)\times(1,\infty))\;\mathrm dx
        \geq   \int_{1}^\infty \nu_0((x,\infty)\times(x,\infty))\;\mathrm dx= \nu_0((1,\infty)\times(1,\infty))\int_{1}^\infty x^{-\alpha_0}\, \mathrm dx>0.
\end{eqnarray*}
Hence $0<\int_{1}^{\infty}  \nu_0((x,\infty)\times(1,\infty))\;\mathrm dx < \infty$.  Therefore, since $t=b_{2}(1/p)\uparrow \infty$ as $p\downarrow0$, we have

\begin{align*}
\lim_{p\downarrow0}\frac{pb_0^{\la}(b_2(1/p))}{b_2(1/p)} \MME(p)  & = \lim_{t\to\infty}  \int_{1}^{\infty} \nu_{t}(x)\; \mathrm dx\\
												      & = \lim_{t\to\infty}  \left[\int_{1}^{M} \nu_{t}(x)\; \mathrm dx +  \int_{M}^{\infty} \nu_{t}(x)\; \mathrm dx \right]\\
												      & = \lim_{M\to \infty} \left[\lim_{t\to\infty} \int_{1}^{M} \nu_{t}(x)\; \mathrm dx +  \lim_{t\to\infty}\int_{M}^{\infty} \nu_{t}(x)\; \mathrm dx \right] \quad \text{(since it is true for any $M\ge 1$)}\\
												      & =  \lim_{M\to \infty} \int_{1}^{M} \nu_{0}((x,\infty)\times(1,\infty)) \;\mathrm dx+ \lim_{M\to \infty} \psi_{M} \quad\quad(\text{using }\eqref{DCT}) \\
                                                                                                          &  =  \int_{1}^{\infty} \nu_{0}((x,\infty)\times(1,\infty))\; \mathrm dx \quad\quad(\text{using } \eqref{psiM}).
\end{align*}                                                                                                           

\end{proof}

\begin{Corollary} \label{Corollary 3.3}
Suppose $\bZ=(Z_1,Z_2)$  satisfies Assumptions \ref{cond:basic}, (B1) and $\overline F_{Z_2}\in\RV_{-\beta}$ for
some $\alpha\leq\beta\leq\alpha_0$. Then $\MME(1/t)\in \RV_{(1+\beta-\alpha_0)/\beta}$.
For $1+\beta>\alpha_0$ we have $\lim_{p\to 0}\MME(p)=\infty$ with $$\frac{1+\beta-\alpha_0}{\beta}\in\left[1-\frac{\alpha_0-1}{\beta},\frac{1}{\alpha_0}\right]\subseteq \left(0,1\right]$$ and for
$1+\beta<\alpha_0$ we have $\lim_{p\to 0}\MME(p)=0$.
\end{Corollary}

\begin{Remark}  \label{Remark 3.3}
A few consequences of   \Cref{Corollary 3.3} are illustrated below.
\begin{itemize}
    \item[(a)] When $1+\beta>\alpha_0$, although the quantity $\MME(p)$ increases as $p\downarrow 0$, the rate of increase is slower than a linear function.
    \item[(b)]
        Let $\bZ\in \MRV(\alpha, b,\nu, \Esp)$.
        Suppose $Z_1$ and $Z_2$ are independent and $\ov F_{Z_1}\in\RV_{-\alpha}$ then by Karamata's Theorem,
        \begin{eqnarray*}
            \MME(p)\sim \frac{1}{\alpha-1}\VaR_{1-p}(Z_2)\P(Z_1>\VaR_{1-p}(Z_2)) \quad (p\downarrow 0).
        \end{eqnarray*}
        This is a special case of \cref{thm:MME_AI}.
\end{itemize}
\end{Remark}

\begin{Example} $\mbox{}$ \label{Example 4.3}
In this example we illustrate the influence of the tail behavior of the marginals as well as the dependence
structure  on the asymptotic behavior of the $\MME$. Assume that $\bZ=(Z_1,Z_2) \in [0,\infty)^{2}$
satisfies Assumptions (A1)-(A4). We compare the following tail independent and tail dependent models:
\begin{itemize}
    \item[(D)] Tail dependent model: Additionally $\bZ$ is tail dependent implying  $R\neq 0$ and satisfies \eqref{i:2}.
        We denote its Marginal Mean Excess  by   $\MME^D$. 
        \item[(ID)] Tail independent model: Additionally $\bZ$ is \emph{asymptotically} tail independent satisfying (A5), (B1) and $1+\beta>\alpha_0>\alpha$.
        Its Marginal Mean Excess we denote by
        $\MME^I$.
\end{itemize}
\begin{enumerate}[(a)]
\item Suppose $Z_1,Z_2$ are identically distributed. Since $t/b_0^{\la}(b_2(t))\in\RV_{1-\alpha_0/\alpha}$ and
$1-\alpha_0/\alpha < 0$ we get
\begin{eqnarray*}
    \frac{\MME^{I}(p)}{\MME^D(p)} \underset{p\to0}{\sim}
    \frac{C}{pb_0^{\la}(b_2(1/p))} \underset{p\to0}{\to}  0.
\end{eqnarray*}
 This means in the asymptotically tail
independent case the Marginal Mean Excess increases at a slower rate to infinity, than in the asymptotically tail
dependent case, as  expected.
\item Suppose $Z_1, Z_2$ are not identically distributed and for some finite constant $C>0$
\begin{eqnarray*}
    \P(Z_2>t)\sim C\P(Z_1>t,Z_2>t)\quad (t\to\infty).
\end{eqnarray*}
This means that not only $\bZ \in \MRV({\alpha_0}, b_0, \nu_0, \Esp_0)$ but also $Z_2\in \RV_{{-\alpha_{0}}}$ and
$Z_1$ is heavier tailed than $Z_2$. Then
\begin{eqnarray*}
    \lim_{t\to\infty}\frac{b_0^{\la}(b_2(t))}{t}=\lim_{t\to\infty}\frac{1}{t\P(Z_1>b_2(t),Z_2>b_2(t))}
    =\lim_{t\to\infty}\frac{C}{t\P(Z_2>b_2(t))}=C.
\end{eqnarray*}
Thus,
\begin{align*}
    \lim_{p\to 0}\frac{1}{\VaR_{1-p}(Z_2)} \MME^{I}(p)=\lim_{p\to 0}\frac{ \MME^{I}(p)}{b_2(1/p)}= C\int_0^{\infty} \nu_0((x,\infty)\times(1,\infty)) \;\mathrm dx
\end{align*}
and $\MME^{I}(\cdot)$ is regularly varying of index $-\frac{1}{\alpha_0}$ at $0$.
In this example $Z_2$ is lighter tailed than $Z_1$, and hence, once again we find that in the asymptotically tail independent case
the Marginal Mean Excess $\MME^I$ increases at a slower rate  to infinity than the Marginal Mean Excess $\MME^{D}$  in the asymptotically tail dependent case.
\end{enumerate}
\end{Example}

\subsection{Asymptotic behavior of the MES} \label{subsec:asymptotic:MES}

Here we derive analogous results for the Marginal Expected Shortfall.

\begin{Theorem}\label{thm:MES}
Suppose $\bZ=(Z_1,Z_2)$  satisfies Assumptions \ref{cond:basic}  and
\ref{cond:dct1}. Then
\begin{align}\label{conv:main}
\lim\limits_{p\downarrow 0} \frac{pb_0^{\la}(\VaR_{1-p}(Z_2))}{\VaR_{1-p}(Z_2)} \MES(p)=\lim\limits_{p\downarrow 0} \frac{pb_0^{\la}(b_2(1/p))}{b_2(1/p)} \MES(p)  = \int_0^{\infty} \nu_0((x,\infty)\times(1,\infty))\;\mathrm dx.
\end{align}
Moreover, $0<\int_0^{\infty}\nu_0((x,\infty)\times(1,\infty))\;\mathrm dx<\infty$.
\end{Theorem}

The proof of   \cref{thm:MES} requires  further condition (B2) which can be avoided in \cref{thm:MME_AI}.

\begin{proof}
The proof is similar to that of \cref{thm:MME_AI} which we discussed in detail.
As in \Cref{thm:MME_AI}  we rewrite
\begin{align*}
\frac{pb_0^{\la}(b_2(1/p))}{b_2(1/p)} \MES(p)  & = \frac{\overline{F}_{Z_2}(t)b_0^{\la}(t)}{t} \;\E(Z_1|Z_2>t)
                                 =  \left[\int_0^{1/M}+\int_{1/M}^M+\int_{M}^{\infty}\right] \frac{\P(Z_1>tx, Z_2>t)}{\P(Z_1>t, Z_2>t)} \;\mathrm dx.
							\end{align*}
We can then conclude the statement from (B2) and similar arguments as in the proof of \cref{thm:MME_AI}.
\end{proof}

A similar comparison can be made between the asymptotic behavior of the Marginal Expected Shortfall
for the tail independent and tail dependent case as we have done in \Cref{Example 4.3} for the Marginal Mean Excess.\\

\begin{Remark} \label{Remark:2}
Define $$a(t):=\frac{b_0^{\la}(b_2(t))}{t \,b_2(t)}.$$ Then $\lim_{t\to\infty}a(t)= 0$ is equivalent to $$\lim_{t\to\infty}\frac{\P(Z_2>t)}{t\P(Z_1>t, Z_2>t)}=0.$$ Hence, a consequence of (B2)
   and  \eqref{B2*} is that  $\lim_{t\to\infty}a(t)=0$ and finally,  $\lim_{p\downarrow 0}\MES(p)=\infty$.
    Again a sufficient assumption for $\lim_{t\to\infty}a(t)=0$ is $\ov F_{Z_2}\in \RV_{-\beta}$ with $\alpha_0<\beta+1$
    and a necessary condition is $\alpha_0\leq \beta+1$  (see~\Cref{Lemma:forAssB}).\\
\end{Remark}
\begin{Remark}\label{remark:3.7}
In this study we have only considered a non-negative random variable $Z_1$ while computing $\MES(p) = \E(Z_1|Z_2>\VaR_{{1-p}}(Z_2))$. For a real-valued random variable $Z_1$, we can represent $Z_1=Z_1^{+}-Z_1^{-}$ where $Z_1^{+}=\max(Z_1,0)$ and $Z_2^{{-}} =\max(-Z_2,0)$. Here both $Z_1^{+}$
and $Z_1^{-}$ are non-negative and hence can be dealt with separately. The limit results will depend on the separate dependence structure and tail behaviors of $(Z_1^{+},Z_2)$ and $(Z_1^{-},Z_2)$. \end{Remark}


\subsection{Illustrative models and examples}  \label{subsec:Models}

We finish this section up with a few models and examples where we can calculate limits for $\MES$ and $\MME$. 
In Sections \ref{subsec:mix} and \ref{subsec:bern} we discuss generative models with sufficient conditions satisfying  Assumptions \ref{cond:basic} and \ref{cond:dct1}.  In Section \ref{subsec:copula} we further discuss two copula models where Theorems \ref{thm:MME_AI} and \ref{thm:MES} can be applied. 

\subsubsection{Mixture representation}\label{subsec:mix}

 First we look at models that are generated in an additive fashion (see \citep{weller:cooley:2014,das:resnick:2015}). We will observe that many models can be generated using the additive technique. 

\begin{modelletter} \label{additive:model}
Suppose $\bZ=(Z_1,Z_2),\bY=(Y_1,Y_2),\bV=(V_1,V_2)$ are random  vectors in $\left[0,\infty\right)^2$ such that $\bZ=\bY+\bV$. Assume the following holds:
\begin{enumerate}[(C1)]
\item $\bY \in \MRV(\alpha,b,\nu,\Esp)$ where $\alpha\geq 1$.
\item $Y_1,Y_2$ are independent random variables.
\item $\ov F_{Y_2}\in\RV_{-\alpha^*}$, $1\leq \alpha\leq \alpha^*$.
\item $\bV \in \MRV(\alpha_0,b_0,\nu_0,\Esp)$ and does not possess asymptotic tail independence where $\alpha\leq \alpha_0$ and
\begin{eqnarray*}
    \lim_{t\to\infty}\frac{\P(\|\bV\|>t)}{\P(\|\bY\|>t)}=0.
\end{eqnarray*}
\item $\bY$ and $\bV$ are independent.
\item $\alpha\leq  \alpha_0< 1+\alpha^*$.
\item $\E|Z_1|<\infty$.
\end{enumerate}
\end{modelletter}


Of course, we would like to know, when  Model C satisfies Assumptions \ref{cond:basic} and \ref{cond:dct1}; moreover, when is
$Z\in\HRV({\alpha_0}, b_0, \nu_0, \Esp_0)$? The next theorem provides
a general result to answer these questions in certain special cases.

\begin{Theorem}  \label{theorem:5.3}
Let  $\bZ=\bY+\bV$ be as in Model~\ref{additive:model}.
Then the following statements hold:
\begin{itemize}
    \item[(a)]  $\bZ \in \MRV({\alpha}, b, \nu, \Esp)\cap \HRV({\alpha_0}, b_0, \nu_0, \Esp_0)$ and satisfies \Cref{cond:dct1}.
    \item[(b)] Suppose $Y_1=0$. Then  $(Z_1,Z_1+Z_2)\in \MRV({\alpha}, b, \nu, \Esp)\cap \HRV({\alpha_0}, b_0, \nu_0^+, \Esp_0)$ with
    \begin{eqnarray*}
        \nu_0^{+}(A)&=&\nu_0(\{(v_1,v_2)\in\Esp_0:(v_1,v_1+v_2)\in A\}) \quad \text{ for }A\in\mathcal{B}(\Esp_0)
    \end{eqnarray*}
    and satisfies \Cref{cond:dct1}.
    \item[(c)] Suppose $\liminf_{t\to\infty} \P(Y_1>t)/\P(Y_2>t)>0$. Then $(Z_1,\min(Z_1,Z_2))\in \MRV({\alpha}, b, \nu^{\text{min}}, \Esp)\cap \HRV({\alpha_0}, b_0, \nu_0^{\text{min}}, \Esp_0)$ with
        \begin{eqnarray*}
        \nu^{\text{min}}(A)&=&\nu(\{(y_1,0)\in\Esp:(y_1,0)\in A\}) \quad \text{ for }A\in\mathcal{B}(\Esp),\\
        \nu_0^{\text{min}}(A)&=&\nu_0(\{(v_1,v_2)\in\Esp_0:(v_1,\min(v_1,v_2))\in A\}) \quad \text{ for }A\in\mathcal{B}(\Esp_0)
    \end{eqnarray*}
    and satisfies \Cref{cond:dct1}.
    \item[(d)] Suppose $Y_1=0$. Then $(Z_1,\max(Z_1,Z_2))\in \MRV({\alpha}, b, \nu, \
    E)\cap \HRV({\alpha_0}, b_0, \nu_0^{\text{max}}, \Esp_0)$ with
    \begin{eqnarray*}
        \nu_0^{\text{max}}(A)&=&\nu_0(\{(v_1,v_2)\in\Esp_0:(v_1,\max(v_1,v_2))\in A\}) \quad \text{ for }A\in\mathcal{B}(\Esp_0)
    \end{eqnarray*}
    and satisfies \Cref{cond:dct1}.
    \end{itemize}
\end{Theorem}
For a proof of this theorem we refer to \cite{Fasen:Das:2017}.
\begin{Remark} \label{remark 2.14}
 Note that, in a systemic risk context where the entire system consists of two institutions with 
 risks $Z_{1}$ and $Z_{2}$, the above theorem addresses the variety of ways a systemic risk model can be constructed. If risk is just additive we could refer to part (b), if the system is at risk when both institutions are at risk then we can refer to part (c) and if the global risk is connected to any of the institutions being at risk then we can refer to the model in part (d). Hence, many kinds of models for calculating systemic risk can be obtained under such a model assumption.
\end{Remark}

\subsubsection{Bernoulli model}\label{subsec:bern}

Next we investigate an example generated by using a mixture method for getting hidden regular variation in a non-standard regularly varying model (see \citep{das:mitra:resnick:2013}).

\begin{Example}\label{ex:ex1} Suppose $X_1,X_2,X_3$ are independent  Pareto random variables with parameters $\alpha$, $\alpha_0$ and $\gamma$, respectively, where $1<\alpha<\alpha_0<\gamma$ and $\alpha_0<1+\alpha$.
Let $B$ be a Bernoulli$(q)$ random variable  with $0<q<1$ and  independent of $X_1,X_2,X_3$. Now define
\[\bZ=(Z_1, Z_2) = B(X_1,X_3) + (1-B)(X_2,X_2).\]
This is a popular example, see \cite{resnick:2002,maulik:resnick:2005,das:resnick:2015}.  Note that
\begin{eqnarray*}
    \P(\max(Z_1,Z_2)>t)\sim qt^{-\alpha}\quad \mbox{ and }\quad \P(\min(Z_1,Z_2)>t)\sim \P(Z_2>t)\sim (1-q)t^{-\alpha_0} \quad(t\to\infty),
\end{eqnarray*}
so that $b(1/p)\sim q^{\frac{1}{\alpha}}p^{-\frac{1}{\alpha}}$, $b_0(1/p)\sim b_2(1/p)\sim (1-q)^{\frac{1}{\alpha_0}} p^{-\frac{1}{\alpha_0}}$ as $p\downarrow 0$.
We denote by $\epsilon_x$, the Dirac measure at point $x$. Note that the limit
measure on $\Esp$ concentrates on the two axes. We will look at usual MRV which is given on $\Esp$ by
$$t\,\P\left( \left(\frac{Z_1}{b(t)}, \frac{Z_2}{b(t)} \right) \in \dx \,\dy \right)
 \stackrel{\M}{\to}   \alpha x^{-\alpha-1}\dx \cdot \epsilon_0 (\dy)  =:\nu(\dx\; \dy) \quad (t\to \infty) \quad \mbox{ in }\M(\mathbb{E}), $$
where the limit measure lies on the x-axis. Hence, we seek HRV in the next step on $\Esp\backslash\{x\text{-axis}\}=\left[0,\infty\right)\times(0,\infty)$ and get
$$t\,\P\left( \left(\frac{Z_1}{b_0(t)}, \frac{Z_2}{b_0(t)} \right) \in \dx \,\dy\right)
 \stackrel{\M}{\to}  \alpha_0x^{-\alpha_0-1}\dx \cdot \epsilon_x (\dy) =:\nu_0(\dx\; \dy) \quad (t\to \infty) \quad \mbox{ in }\M(\Esp\backslash\{x\text{-axis}\}).$$
Here  the limit measure lies on the diagonal where $x=y$. Thus, we have for any $x\ge 1$,
\[\nu_0((x,\infty)\times(1,\infty)) =   x^{-\alpha_0}.
  \]
Now, we can explicitly calculate the values of MME and MES. For $0<p<1$:
\begin{align*}
\MES(p)
    &= \frac{1}{q\VaR_{1-p}(Z_2)^{-\gamma}+(1-q)\VaR_{1-p}(Z_2)^{-\alpha_0}} \left[ \frac{q\alpha}{\alpha-1}\VaR_{1-p}(Z_2)^{-\gamma} + \frac{(1-q)\alpha_0}{\alpha_0-1}\VaR_{1-p}(Z_2)^{-\alpha_0+1}\right],\\
\MME(p)
    &= \frac{1}{q\VaR_{1-p}(Z_2)^{-\gamma}+(1-q)\VaR_{1-p}(Z_2)^{-\alpha_0}} \left[ \frac{q}{\alpha-1}\VaR_{1-p}(Z_2)^{-\gamma-\alpha+1} + \frac{(1-q)}{\alpha_0-1}\VaR_{1-p}(Z_2)^{-\alpha_0+1}\right].
\end{align*}
  Therefore,
    \begin{align*}
  \frac{pb_0^{\la}(b_2(1/p))}{b_2(1/p)} \MME(p) & \sim   \frac{1}{b_2(1/p)} \MME(p)\sim \frac{1}{\alpha_0-1} = \int_1^{\infty} \nu_0((x,\infty)\times(1,\infty)) \, \dx \quad (p\downarrow 0),
  \end{align*}
and
  \begin{align*}
  \frac{pb_0^{\la}(b_2(1/p))}{b_2(1/p)} \MES(p) & \sim   \frac{1}{b_2(1/p)} \MES(p)\sim \frac{\alpha_0}{\alpha_0-1} = \int_0^{\infty} \nu_0((x,\infty)\times(1,\infty)) \, \dx
  \quad (p\downarrow 0).
  \end{align*}
\end{Example}

\subsubsection{Copula models} \label{subsec:copula}
The next two examples constructed by well-known copulas (see \cite{Nelsen}) are illustrative of the limits which we are able to compute using Theorems \ref{thm:MME_AI} and \ref{thm:MES}.
\begin{Example}\label{example:gausscop}
 In financial risk management, no doubt the most famous copula model is the \emph{Gaussian copula}:
\begin{eqnarray*}
    C_{\Phi,\rho}(u,v)=\Phi_2(\Phi^{\leftarrow}(u),\Phi^{\leftarrow}(v)) \quad \text{for } (u,v)\in[0,1]^2,
\end{eqnarray*}
where $\Phi$ is the standard-normal distribution function and $\Phi_2$ is a bivariate normal distribution function
with standard normally distributed margins and correlation $\rho$. Then  the survival copula  satisfies:
\begin{eqnarray*}
    \wh C_{\Phi,\rho}(u,u)=C_{\Phi,\rho}(u,u)\sim u^{\frac{2}{\rho+1}}\ell(u) \quad (u\to0),
\end{eqnarray*}
for some function $l$ which is slowly varying at 0, see \cite{reiss:1989,Ledford:Tawn}.
Suppose $(Z_1,Z_2)$  has identical Pareto marginal distributions with common parameter $\alpha>0$  and  a dependence structure given by a  Gaussian copula
$ C_{\Phi,\rho}(u,v)$ with $\rho\in(-1,1)$.
Now we can check that  $(Z_1,Z_2) \in \MRV(\alpha,b,\nu,\mathbb{E})$ with asymptotic tail independence and $(Z_1,Z_2)\in \MRV(\alpha_0,b_0,\nu_0,\mathbb{E}_0)$ with
\begin{align*}
\alpha_0 =\frac{2\alpha}{1+\rho} \quad \mbox{ and } \quad
   \nu_0((x,\infty)\times(y,\infty)) =           x^{-\frac{\alpha}{1+\rho}}y^{-\frac{\alpha}{1+\rho}}, \quad x,y>0 .
 \end{align*}
        Hence, for $\rho\in(1-2/(\alpha+1),1)$  we have $\lim_{p\to 0}\MME(p)=\infty$. In this model, Assumptions \ref{cond:basic} and (B1) are satisfied when  $\alpha>1+\rho$ and $\alpha>1$.
We can also check that Assumption (B2) is not satisfied.
Consequently, we can find estimates for $\MME$ but not for $\MES$ in this example.
\end{Example}
\begin{Example}\label{example:marolcop}
 Suppose $(Z_1,Z_2)$  has identical Pareto marginal distributions with parameter $\alpha>0$  and  a dependence structure
    given by a    \emph{Marshall-Olkin survival copula}:
$$\wh C_{\gamma_1,\gamma_2}(u,v)=uv\min(u^{-\gamma_1},v^{-\gamma_2})  \quad \text{for } (u,v)\in[0,1]^2,$$ for some $\gamma_1,\gamma_2\in(0,1)$.
We can check that in this model, we have $(Z_1,Z_2) \in \MRV(\alpha,b,\nu,\Esp)$ with asymptotic tail independence and $(Z_1,Z_2)\in \MRV(\alpha_0,b_0,\nu_0,\Esp_{0})$ with
\begin{align*}
\alpha_0 &= \alpha\max(2-\gamma_1,2-\gamma_2) \quad \text{ and }\;\;\;\\
  \nu_0((x,\infty)\times(y,\infty)) &= \left\{ \begin{array}{ll}
                                    x^{-\alpha(1-\gamma_1)}y^{-\alpha},  & \gamma_1<\gamma_2,\\
                                    x^{-\alpha}y^{-\alpha}\max(x,y)^{-\alpha\gamma_1},    & \gamma_1=\gamma_2,\\
                                    x^{-\alpha}y^{-\alpha(1-\gamma_2)},  & \gamma_1>\gamma_2,\\
                                     \end{array} \right.    \quad x,y>0.
 \end{align*}
         Then $\min(\gamma_1,\gamma_2)\in(1-1/\alpha,1)$ implies $\lim_{p\to 0}\MME(p)=\infty$.   Moreover this model satisfies Assumptions \ref{cond:basic} and (B1) when $\gamma_1\ge\gamma_2$.  Unfortunately again, (B2) is not satisfied.
\end{Example}   


\section{Estimation of  MME and MES} \label{sec:estimation}

\subsection{Empirical estimators  for the MME and the  MES} \label{subsec:empirical}

\subsubsection{Empirical estimator for the MME}
Suppose $(Z^{(1)}_{1},Z^{(2)}_{1}), \ldots, (Z^{(1)}_{n},Z^{(2)}_{n}) $ are iid samples with the same distribution as  $(Z_{1},Z_{2})$.
We denote by  $Z_{(1:n)}^{(2)}\geq \ldots\geq Z_{(n:n)}^{(2)}$ the order statistic of the  sample $Z_1^{(2)},\ldots,Z_n^{(2)}$ in decreasing order. We begin by looking at the behavior of the empirical estimator
\begin{eqnarray*}
    \wh{\MME}_{\text{emp},n}\left(k/n\right):=\frac{1}{k}\sum_{i=1}^n (Z_i^{(1)}-Z^{(2)}_{(k:n)})_+\1_{\{Z_i^{(2)}>Z^{(2)}_{(k:n)}\}}
\end{eqnarray*}
of the quantity $\MME(k/n)=\E((Z_1-b_2(n/k))_+|Z_2>b_2(n/k))$ with $k<n$. The following theorem shows that the empirical estimator is consistent in probability.

\begin{Proposition} \label{prop:empirical:MME}
Let the assumptions of \cref{thm:MME_AI} hold, and let $\ov F_{Z_2}\in\RV_{-\beta}$ for some $\alpha\leq\beta\leq\alpha_0$. Furthermore, let $k=k(n)$ be a sequence of
integers satisfying $k\to\infty$, $k/n\to 0$ and $b_0^\la(b_2(n/k))/n\to 0$ as $n\to\infty$ (note that this is trivially satisfied if $b_0=b_2$). 
\begin{itemize}
\item[(a)] Then, as $n\to\infty$,
\begin{eqnarray*}
    \frac{b_0^{\leftarrow}(b_2(n/k))}{b_2(n/k)}\frac{1}{n}\sum_{i=1}^n (Z_i^{(1)}-Z^{(2)}_{(k:n)})_+\1_{\{Z_i^{(2)}>Z^{(2)}_{(k:n)}\}}
    \stackrel{P}{\to}\int_1^{\infty}\nu_0((x,\infty)\times(1,\infty))\, \mathrm dx.
\end{eqnarray*}
\item[(b)] In particular, we have
${\displaystyle
    \frac{\wh{\MME}_{\text{\rm emp},n}\left(k/n\right)}{\MME(k/n)}\stackrel{P}{\to}1}$ as $n\to\infty$.
\end{itemize}
\end{Proposition}
To prove this theorem we use the following lemma.

\begin{Lemma} \label{Lemma:empirical:MME}
Let the assumptions of \Cref{prop:empirical:MME} hold. Define for $y>0$,
\begin{eqnarray*}
    E_n(y)&:=&\frac{b_0^{\leftarrow}(b_2(n/k))}{b_2(n/k)}\frac{1}{n}\sum_{i=1}^n (Z_i^{(1)}-b_2(n/k)y)_+\1_{\{Z_i^{(2)}>b_2(n/k)y\}},\\
    E(y)&:=&\int_y^\infty\nu_0((x,\infty)\times(y,\infty))\,\mathrm dx.
\end{eqnarray*}
Then $E(y)=y^{1-\alpha_0}E(1)$ and as $n\to\infty$,
\begin{eqnarray*} 
(E_n(y))_{y\geq 1/2}\stackrel{{P}}{\to}(E(y))_{y\geq 1/2} \quad \mbox{ in }\mathbb{D}(\left[{1}/{2},\infty\right),(0,\infty)),
\end{eqnarray*}
where by $\mathbb{D}(I,\mathbb{E^*})$ we denote the space of  c\`{a}dl\`{a}g functions from $I\to\mathbb{E^*}.$
\end{Lemma}
 
\begin{proof} $\mbox{}$\\
We already know from \cite[Theorem 5.3(ii)]{Resnick:2007}, $\bZ \in \MRV({\alpha_0}, b_0, \nu_0, \Esp_0)$ and \linebreak
$b_0^\la(b_2(n/k))/n\to 0$ that as $n\to\infty$,
\begin{eqnarray} \label{C.2}
    \nu_0^{(n)}:=\frac{b_0^\la(b_2(n/k))}{n}\sum_{i=1}^n\epsilon_{\left(\frac{Z_i^{(1)}}{b_2(n/k)},\frac{Z_i^{(2)}}{b_2(n/k)}\right)}
        \Rightarrow \nu_0 \quad \mbox{ in }\mathbb{M}_+(\mathbb{E}_0).
\end{eqnarray}
Note that
\begin{eqnarray*}
    E_n(y)=\int_y^\infty\nu_0^{(n)}((x,\infty)\times(y,\infty))\,\mathrm dx
        =\frac{b_0^\la(b_2(n/k))}{b_2(n/k)}\frac{1}{n}\sum_{i=1}^n (Z_i^{(1)}-b_2(n/k)y)_+\1_{\{Z_i^{(2)}>b_2(n/k)y\}}.
\end{eqnarray*}
Hence, the statement of the lemma is equivalent to
\begin{eqnarray} \label{C.3}
    \left(\int_y^\infty\nu_0^{(n)}((x,\infty)\times(y,\infty))\,\mathrm dx\right)_{y\geq\frac{1}{2}}\stackrel{{P}}{\to}(E(y))_{y\geq 1/2}
    \mbox{ in }\mathbb{D}(\left[{1}/{2},\infty\right),(0,\infty)).
\end{eqnarray}
We will prove \eqref{C.3} by a convergence-together argument.

\noindent \textbf{Step 1.} First we prove that $E(y)=y^{1-\alpha_0}E(1)$. Note that
\begin{align} \label{eq5.3}
    \lefteqn{\frac{b_0^{\leftarrow}(b_2(n/k))}{b_2(n/k)}\E((Z_1-b_2(n/k)y)_+\1_{\{Z_2>b_2(n/k)y\}})} \nonumber\\
        &=\int_1^\infty\frac{\P(Z_1>xb_2(n/k),Z_2>b_2(n/k)y)}{\P(Z_1>b_2(n/k),Z_2>b_2(n/k))}\,\mathrm dx \nonumber\\
        &=y\frac{\P(Z_1>b_2(n/k)y,Z_2>b_2(n/k)y)}{\P(Z_1>b_2(n/k),Z_2>b_2(n/k))}\int_1^\infty\frac{\P(Z_1>x(b_2(n/k)y),Z_2>b_2(n/k)y)}{\P(Z_1>b_2(n/k)y,Z_2>b_2(n/k)y)}\,\mathrm dx \nonumber\\
         &=y \cdot \frac{\P(Z_1>b_2(n/k)y,Z_2>b_2(n/k)y)}{\P(Z_1>b_2(n/k),Z_2>b_2(n/k))} \cdot \int_1^\infty\nu_{t} (x) \,\mathrm dx \nonumber\\ \intertext{(where $\nu_{t}$ is as defined in \eqref{def:nut} with $t=b_{2}(n/k)y$ )}
        &\stackrel{n\to\infty}{\to} y\cdot y^{-\alpha_0} \cdot \int_1^\infty\nu_0((x,\infty)\times(1,\infty))\,\mathrm dx=y^{1-\alpha_0}E(1).
\end{align}
The final limit follows from the definition of hidden regular variation and Theorem \ref{thm:MME_AI}.
On the other hand, in a similar manner as in \Cref{thm:MME_AI}, we can exchange the integral  and the
limit such that using \eqref{eqn:ccc} we obtain
\begin{eqnarray}  \label{eq5.4}
    \frac{b_0^{\leftarrow}(b_2(n/k))}{b_2(n/k)}\E((Z_1-b_2(n/k)y)_+\1_{\{Z_2>b_2(n/k)y\}}) \nonumber
        &=&\int_1^\infty\frac{\P(Z_1>x b_2(n/k),Z_2>b_2(n/k)y)}{\P(Z_1>b_2(n/k),Z_2>b_2(n/k))}\,\mathrm dx\\
        &\stackrel{n\to\infty}{\to}&\int_1^\infty\nu_0((x,\infty)\times(y,\infty))\,\mathrm dx=E(y).
\end{eqnarray}
Since \eqref{eq5.3} and \eqref{eq5.4} must be equal, we have $E(y)=y^{1-\alpha_0}E(1)$.

\noindent \textbf{Step 2.}
Now we prove that for any $y\geq 1/2$ and $M>0$,  as $n\to\infty$,
\begin{eqnarray} \label{C.4}
    E_n^{(M)}(y):=\int_{1}^M \nu_0^{(n)}((x,\infty)\times(y,\infty))\,\mathrm dx
    \stackrel{{P}}{\to}\int_{1}^M\nu_0((x,\infty)\times(y,\infty))\,\mathrm dx=:E^{(M)}(y).
\end{eqnarray}
Define the function $f_{M,y}:\Esp_0\to\left[0,M\right]$ as $f_{M,y}(z_1,z_2)=(\min\left(z_1,M\right)-y)\1_{\{z_1>y,z_2>y\}}$ which is continuous,
bounded and has compact support on $\Esp_0$ and for any $y\geq 1/2$ define $F_{M,y}:\mathbb{M}_+(\mathbb{E}_0)\to \R_+$ as
$$m\mapsto \int_{\mathbb{E}_0} f_{M,y}(z_1,z_2)\,m(dz_1,d z_2).$$
Here $m$ is a continuous map  on $\mathbb{M}_+(\mathbb{E}_0)$ under the vague topology.  Hence, using a continuous mapping theorem and \eqref{C.2} we get, as $n\to\infty$,
\begin{eqnarray} \label{5.4}
    \int_{1}^M \nu_0^{(n)}((x,\infty)\times(y,\infty))\,\mathrm dx=F_{M,y}(\nu_0^{(n)})
        \weak F_{M,y}(\nu_0) =\int_{1}^M\nu_0((x,\infty)\times(y,\infty))\,\mathrm dx
\end{eqnarray}
in $\R_+$. Since the right hand side is deterministic, the convergence holds in probability as well.

\noindent \textbf{Step 3.} Using  Assumption (B1),
\begin{eqnarray*}
    \E\left(\sup_{y\geq\frac{1}{2}}\int_M^{\infty}\nu_0^{(n)}((x,\infty)\times(y,\infty))\,\mathrm dx\right)
    & = & \E\left(\int_M^{\infty}\nu_0^{(n)}((x,\infty)\times(1/2,\infty))\,\mathrm dx\right)\\
    &=&{b_0^\la(b_2(n/k))}\int_M^{\infty}\P(Z_1>xb_2(n/k),Z_2>b_2(n/k)/2)\,\mathrm dx\\
    &=&\int_M^{\infty}\frac{\P(Z_1>xb_2(n/k),Z_2>b_2(n/k)/2)}{\P(Z_1>b_2(n/k),Z_2>b_2(n/k))}\,\mathrm dx\stackrel{n\to\infty,M\to\infty}{\to}0.
\end{eqnarray*}

\noindent \textbf{Step 4.}
Hence, a convergence-together argument (see \cite[Theorem 3.5]{Resnick:2007}), Step 2, Step 3 and $E^{(M)}(y)\to E(y)$ as $M\to\infty$ result in
$E_n(y)\stackrel{P}{\to}E(y)$ as $n\to\infty$.

\noindent \textbf{Step 5.} From Step 1,  the function $E:\left[1/2,\infty\right)\to\left(0,E(1/2)\right]$ is a decreasing, continuous function
as well as a bijection. Let $E^{-1}$ denote its inverse and define for $m\in\N$ and $k=1,\ldots,m$,
$$y_{m,k}:=E^{-1}\left(E(1/2)\frac{k}{m}\right).$$
As in the proof of the Glivenko-Cantelli-Theorem (see \cite[Theorem 20.6]{Billingsley:1995}) we have
\begin{eqnarray*}
    \sup_{y\geq 1/2}|E_n(y)-E(y)|\leq\frac{E(1/2)}{m}+\sup_{k=1,\ldots,m}|E_n(y_{m,k})-E(y_{m,k})|.
\end{eqnarray*}
Let $\epsilon>0$. Choose $m\in\N$ such that $m>2E(1/2)/\epsilon$. Then
\begin{eqnarray*}
    \P\left(\sup_{y\geq 1/2}|E_n(y)-E(y)|>\epsilon\right)&\leq&
    \P\left(\sup_{k=1,\ldots,m}|E_n(y_{m,k})-E(y_{m,k})|>E(1/2)m^{-1}\right)\\
    &\leq&\sum_{k=1}^m
    \P\left(|E_n(y_{m,k})-E(y_{m,k})|>E(1/2)m^{-1}\right)\stackrel{n\to\infty}{\to}0,
\end{eqnarray*}
where we used $E_n(y_{m,k})\stackrel{P}{\to}E(y_{m,k})$ as $n\to\infty$ for any $k=1,\ldots,m,\,m\in\N$ by Step 4.
Hence, we can conclude the statement.
\end{proof}
\hspace*{-0.5cm}\textbf{Proof of \Cref{prop:empirical:MME}.} \\
(a) \, By assumption, $\ov F_{Z_2}\in\RV_{-\beta}$. From \cite[p. 82]{Resnick:2007} we know that
\begin{eqnarray*}
    \left(\frac{Z^{(2)}_{(\lceil ky\rceil:n) }}{b_2(n/k)}\right)_{y> 0}\stackrel{P}{\to}\left(y^{-\frac{1}{\beta}}\right)_{y> 0} \quad
   \mbox{ in }\mathbb{D}(\left(0,\infty\right],(0,\infty))
\end{eqnarray*}
and in particular, this and \Cref{Lemma:empirical:MME} result in
\begin{eqnarray*}
    \left((E_n(y))_{y\geq\frac{1}{2}},\left(\frac{Z^{(2)}_{(\lceil ky\rceil:n) }}{b_2(n/k)}\right)_{y> 0}\right)\stackrel{P}{\to}\left((E(y))_{y\geq\frac{1}{2}},(y^{-\frac{1}{\beta}})_{y>0 }\right) \mbox{ in }\mathbb{D}(\left[1/2,\infty\right),(0,\infty))\times\mathbb{D}(\left(0,\infty\right],(0,\infty)).
\end{eqnarray*}
Let $\mathbb{D}^\downarrow(\left(0,2^\beta\right],\left[1/2,\infty\right))$ be a subfamily of $\mathbb{D}(\left(0,2^\beta\right],\left[1/2,\infty\right))$ consisting
of non-increasing functions. Let us similarly define $\mathbb{C}^\downarrow(\left(0,2^\beta\right],\left[1/2,\infty\right))$. Define
the map $\varphi:\mathbb{D}(\left[1/2,\infty\right),(0,\infty))\times \mathbb{D}^\downarrow(\left(0,2^\beta\right],\left[1/2,\infty\right))$ with $(f,g)\mapsto f\circ g$. From \cite[Theorem 13.2.2]{Whitt2002}, we already know that $\varphi$ restricted to $\mathbb{D}(\left[1/2,\infty\right),(0,\infty))\times \mathbb{C}^\downarrow(\left(0,2^\beta\right],\left[1/2,\infty\right))$ is continuous.
Thus, we can apply a continuous mapping theorem and obtain as $n\to\infty$,
\begin{eqnarray*}
    \left(E_n\left(\frac{Z^{(2)}_{(\lceil ky\rceil :n) }}{b_2(n/k)}\right)\right)_{y\in\left(0,2^{\beta}\right]}\stackrel{P}{\to}\left(E(y^{-\frac{1}{\beta}})\right)_{y\in\left(0,2^{\beta}\right]} \quad
    \mbox{ in } \mathbb{D}(\left(0,2^{\beta}\right],(0,\infty)).
\end{eqnarray*}
As a special case we get the marginal convergence as $n\to\infty$,
\begin{eqnarray*}
      \frac{b_0^{\leftarrow}(b_2(n/k))}{b_2(n/k)}\frac{1}{n}\sum_{i=1}^n Z_i^{(1)}\1_{\{Z_i^{(2)}>Z^{(2)}_{( k :n) }\}}&=&E_n\left(\frac{Z^{(2)}_{(k :n) }}{b_2(n/k)}\right)
      \stackrel{P}{\to}E(1)=\int_0^{\infty}\nu_0((x,\infty)\times(1,\infty))\,\mathrm dx.
\end{eqnarray*}

(b) \, Finally, from part (a) and \Cref{thm:MME_AI} we have
\begin{eqnarray*}
    \frac{\frac{1}{k}\sum_{i=1}^n Z_i^{(1)}\1_{\{Z_i^{(2)}>Z^{(2)}_{(k:n)}\}}}{\MME(k/n)}
        =\frac{\frac{b_0^{\leftarrow}(b_2(n/k))}{b_2(n/k)}\frac{1}{n}\sum_{i=1}^n Z_i^{(1)}\1_{\{Z_i^{(2)}>Z^{(2)}_{(k:n)}\}}}{
        \frac{\frac{k}{n}b_0^{\leftarrow}(b_2(n/k))}{b_2(n/k)}\MME(k/n)}
    \stackrel{P}{\to}\frac{\int_0^{\infty}\nu_0((x,\infty)\times(1,\infty))\,\mathrm dx}{\int_0^{\infty}\nu_0((x,\infty)\times(1,\infty))\,\mathrm dx}=1,
\end{eqnarray*}
which is what we needed to show.
\hfill$\Box$

\subsubsection{Empirical estimator for the MES}

An analogous result holds for the empirical estimator
\begin{eqnarray*}
    \wh{\MES}_{\text{emp},n}\left(k/n\right):=\frac{1}{k}\sum_{i=1}^n Z_i^{(1)}\1_{\{Z_i^{(2)}>Z^{(2)}_{(k:n)}\}}
\end{eqnarray*}
of $\MES(k/n)=\E(Z_1|Z_2>b_2(n/k))$ where $k<n$.

\begin{Proposition} \label{prop:empirical}
Let the assumptions of \cref{thm:MES} hold, and let $\ov F_{Z_2}\in\RV_{-\beta}$ for some $\alpha\leq\beta\leq\alpha_0$. Furthermore, let $k=k(n)$ be a sequence of
integers satisfying $k\to\infty$, $k/n\to 0$ and $b_0^\la(b_2(n/k))/n\to 0$ as $n\to\infty$.
\begin{itemize}
\item[(a)] Then, as $n\to\infty$,
\begin{eqnarray*}
    \frac{b_0^{\leftarrow}(b_2(n/k))}{b_2(n/k)}\frac{1}{n}\sum_{i=1}^n Z_i^{(1)}\1_{\{Z_i^{(2)}>Z^{(2)}_{(k:n)}\}}
    \stackrel{P}{\to}\int_0^{\infty}\nu_0((x,\infty)\times(1,\infty))\,\mathrm dx.
\end{eqnarray*}
\item[(b)] In particular,
${\displaystyle
    \frac{\wh{\MES}_{\text{emp},n}\left(k/n\right)}{\MES(k/n)}\stackrel{P}{\to}1}$, as $n\to\infty$
\end{itemize}
\end{Proposition}

The proof of the theorem is analogous to the proof of \Cref{prop:empirical:MME} based on the following  version of \Cref{Lemma:empirical:MME}.
Hence, we skip the details.

\begin{Lemma} \label{Lemma:empirical:MES}
Let the assumptions of \Cref{prop:empirical} hold. Define for $y>0$,
\begin{eqnarray*}
    E^{*}_n(y)&:=&\frac{b_0^{\leftarrow}(b_2(n/k))}{b_2(n/k)}\frac{1}{n}\sum_{i=1}^n Z_i^{(1)}\1_{\{Z_i^{(2)}>b_2(n/k)y\}},\\
    E^{*}(y)&:=&\int_0^\infty\nu_0((x,\infty)\times(y,\infty))\,\mathrm dx.
\end{eqnarray*}
Then $E^{*}(y)=y^{1-\alpha_0}E^{*}(1)$ and as $n\to\infty$,
\begin{eqnarray*}
(E^{*}_n(y))_{y\geq 1/2}\stackrel{{P}}{\to}(E^{*}(y))_{y\geq 1/2} \quad \mbox{ in }\mathbb{D}(\left[{1}/{2},\infty\right),(0,\infty)).
\end{eqnarray*}
\end{Lemma}
\begin{proof}
The only differences between the proofs of \Cref{Lemma:empirical:MME} and \Cref{Lemma:empirical:MES} are that
in the proof of \Cref{Lemma:empirical:MES}  we use $E_n^{*(M)}(y):=\int_{\frac{1}{M}}^M \nu_0^{(n)}((x,\infty)\times(y,\infty))\,\mathrm dx$
and that in Step~3, we have
$$\lim_{M\to\infty}\lim_{n\to\infty}\E\left(\sup_{y\geq\frac{1}{2}}\left[\int_0^{\frac{1}{M}}+\int_M^{\infty}\right]\nu_0^{(n)}((x,\infty)\times(y,\infty))\,\mathrm dx\right)=0$$
where Assumption (B2) has to be used.
\end{proof}

\subsection{Estimators for the MME and the MES based on extreme value theory } \label{subsec:EVT estimators}

In certain situations we might be interested in estimating $\MME(p)$ or $\MES(p)$ in a region
where no data are available. Since empirical estimators would not work in such a case we can resort to extrapolation via extreme value theory. We start with a
motivation for the definition of the estimator before we provide its'  asymptotic properties.
For the rest of this section we  make the following assumption.
\begin{assumptionletter} \label{assumption D} $\mbox{}$
$\ov F_{Z_2}\in\RV_{-\beta}$ for $\alpha\leq\beta\leq\alpha_0<\beta+1$.
\end{assumptionletter}
\Cref{assumption D} guarantees that $\lim_{t\to\infty}a(t)=0$ (see \cref{Remark:2}).
The idea here is that for all $p\geq k/n$, we estimate  $\MME(p)$ empirically since sufficient data are available in this region; on the other hand  for $p < k/n$ we will use an extrapolating extreme-value technique.
For notational convenience, define  the function $$a(t):=\frac{b_0^{\la}(b_2(t))}{t\,b_2(t)}.$$ Since
$b_0^\la\in\RV_{\alpha_0}$ and $b_2\in\RV_{1/\beta}$, we have $a\in\RV_{\frac{\alpha_0-\beta-1}{\beta}}$. Now,
let $k:=k(n)$ be a sequence of integers so that $k/n\to 0$ as $n\to\infty$. From \cref{thm:MME_AI} we already know that
\begin{eqnarray*}
    \lim\limits_{p\downarrow 0} a(1/p) \MME(p)  = \int_1^{\infty} \nu_0((x,\infty)\times(1,\infty))\;\mathrm dx
    =\lim\limits_{n\to\infty} a(n/k) \MME(k/n).
\end{eqnarray*}
Hence,
\begin{eqnarray} \label{MSE_As}
    \MME(p)\sim \frac{a(n/k)}{a(1/p)} \MME(k/n)
        \sim\left(\frac{k}{np}\right)^{\frac{\beta-\alpha_0+1}{\beta}}\MME(k/n) \quad (p\downarrow 0).
\end{eqnarray}
If we plug in the estimators $\wh \alpha_{0,n}$, $\wh\beta_n$ and $\wh \MME_n(k/n)$ for
$\alpha_0$, $\beta$ and $\MME(k/n)$ respectively in \eqref{MSE_As} we
obtain an estimator for $\MME(p)$ given by
\begin{eqnarray*}
    \wh \MME_n(p)=\left(\frac{k}{np}\right)^{\frac{\wh\beta_n-\wh\alpha_{0,n}+1}{\wh\beta_n}}\wh\MME_{\text{emp},n}(k/n).
\end{eqnarray*}
Similarly, we may obtain an estimator of $\MES(p)$ given by
\begin{eqnarray*}
    \wh \MES_n(p)=\left(\frac{k}{np}\right)^{\frac{\wh\beta_n-\wh\alpha_{0,n}+1}{\wh\beta_n}}\wh\MES_{\text{emp},n}(k/n).
\end{eqnarray*}
If $\beta>\alpha$ then the parameter $\alpha$, the index of regular variation of $Z_1$,
is surprisingly not necessary for the estimation of either Marginal Mean Excess  or Marginal Expected Shortfall.

\begin{Theorem}  \label{Theorem:6.3}
Let  Assumptions \ref{cond:basic}, (B1) and \ref{assumption D} hold. Furthermore, let $k=k(n)$
be a sequence of
integers satisfying $k\to\infty$, $k/n\to 0$ as $n\to\infty$. Moreover,
 $p_n\in(0,1)$ is a sequence of constants with
$p_n\downarrow 0$ and $np_n=o(k)$ as $n\to\infty$. Let $\wh \alpha_{0,n}$ and $\wh \beta_n$ be
estimators for $\alpha_0$ and $\beta$, respectively such that
\begin{eqnarray} \label{eq:6.8}
    \ln\left(\frac{k}{np_n}\right)\left(\wh\alpha_{0,n}-\alpha_0\right)\stackrel{P}{\to}0
    \quad \mbox{ and } \quad \ln\left(\frac{k}{np_n}\right)\left(\wh\beta_n-\beta\right)\stackrel{P}{\to}0 \quad (n\to\infty).
\end{eqnarray}
\begin{itemize}
\item[(a)] Then
$
{\displaystyle    \frac{\wh \MME_n(p_n)}{\MME(p_n)}\stackrel{P}{\to}1}
$ as $n\to\infty$.
\item[(b)] Additionally, if Assumption (B2) is satisfied then
$
{\displaystyle    \frac{\wh \MES_n(p_n)}{\MES(p_n)}\stackrel{P}{\to}1}
$ as $n\to\infty$.
\end{itemize}
\end{Theorem}
\begin{proof}
(a) \, Rewrite
\begin{eqnarray*}
    \frac{\wh \MME_n(p_n)}{\MME(p_n)}&=&\frac{\left(\frac{k}{np_n}\right)^{\frac{\wh\beta_n-\wh\alpha_{0,n}+1}{\wh\beta_n}}\wh\MME_{\text{emp},n}(k/n)}{\MME(p_n)}\\
       &=&\frac{\wh\MME_{\text{emp},n}(k/n)}{\MME(k/n)}\frac{a(n/k)\MME(k/n)}{a(1/p_n)\MME(p_n)}
       \frac{\frac{a(1/p_n)}{a(n/k)}}{\left(\frac{np_n}{k}\right)^{\frac{\beta-\alpha_0+1}{\beta}}}\frac{\left(\frac{k}{np_n}\right)^{\frac{\wh\beta_n-\wh\alpha_{0,n}+1}{\wh\beta_n}}}
       {\left(\frac{k}{np_n}\right)^{\frac{\beta-\alpha_0+1}{\beta}}}\\
       &=:&I_1(n)\cdot I_2(n)\cdot I_3(n)\cdot I_4(n).
\end{eqnarray*}
An application of \Cref{prop:empirical:MME} implies $$I_1(n)=\frac{\wh\MME_{\text{emp},n}(k/n)}{\MME(k/n)}\stackrel{P}{\to}1$$ as $n\to\infty$.
For the second term $I_2(n)$, using \cref{thm:MME_AI} we get
\begin{eqnarray*}
    I_2(n)=\frac{\frac{\frac{k}{n}b_0^{\la}(b_2(n/k))}{b_2(n/k)}\MME(k/n)}{\frac{p_nb_0^{\la}(b_2(1/p_n))}{b_2(1/p_n)}\MME(p_n)}
    \stackrel{P}{\to}1 \quad (n\to \infty).
\end{eqnarray*}
Since $a\in\RV_{(\alpha_0-\beta-1)/\beta}$, $k/n\to 0$ and $np_n = o(k)$, we obtain $\lim_{n\to\infty}I_3(n)=1$ as well.
For the last term $I_4(n)$ we use the representation
\begin{eqnarray*}
    \frac{\left(\frac{k}{np_n}\right)^{\frac{\wh\beta_n-\wh\alpha_{0,n}+1}{\wh\beta_n}}}{\left(\frac{k}{np_n}\right)^{\frac{\beta-\alpha_0+1}{\beta}}}
    =\exp\left(\left(\frac{1-\wh\alpha_{0,n}}{\wh\beta_n}-\frac{1-\alpha_0}{\beta}\right)\ln\left(\frac{k}{np_n}\right)\right),
\end{eqnarray*}
and
\begin{eqnarray*}
        \frac{1-\wh\alpha_{0,n}}{\wh\beta_n}-\frac{1-\alpha_{0}}{\beta}=(\beta-\wh\beta_n)\frac{1-\wh\alpha_{0,n}}{\wh\beta_n\beta}+(\alpha_0-\wh\alpha_{0,n})\frac{1}{\beta}.
\end{eqnarray*}
Since by assumption \eqref{eq:6.8} we have $\wh \alpha_{0,n}\stackrel{P}{\to}\alpha$,  $\wh \beta_{n}\stackrel{P}{\to}\beta$, using a continuous mapping theorem we get,
\begin{eqnarray*}
    \ln \left(\frac{k}{np_n}\right) (\beta-\wh\beta_n)\frac{1-\wh\alpha_{0,n}}{\wh\beta_n\beta}+\ln \left(\frac{k}{np_n}\right)(\alpha_0-\wh\alpha_{0,n})\frac{1}{\beta}\stackrel{P}{\to}0.
\end{eqnarray*}
Hence, we conclude
that $I_{4}(n)\stackrel{P}{\to}1$ as $n\to\infty$ which completes the proof.\\
(b) This proof is analogous to (a) and hence is omitted here.
\end{proof}

\section{Simulation study}\label{sec:simulation}

In this section, we study the developed estimators for different models. We simulate from  models described in Section~\ref{sec:prelim} and Section~\ref{sec:asymptotic}, estimate MME and MES values from the data and compare them with the actual values from the model. We also compare our estimator with  a regular empirical estimator and observe that our estimator provides a smaller variance in most simulated examples. Moreover our estimator is scalable to smaller $p<1/n$ where $n$ is the sample size, which is infeasible for the empirical estimator.
\subsection{Estimators and assumption checks}
As an estimator of $\beta$, the index of regular variation of $Z_{2}$ we use the Hill-estimator based on the data $Z_1^{(2)},\ldots,Z_n^{(2)}$ whose order statistics is given by $Z_{(1:n)}^{(2)} \ge \ldots \ge  Z_{(n:n)}^{(2)}$. The estimator is
\begin{eqnarray*}
    \wh\beta_n=\frac{1}{k_2}\sum_{i=1}^{k_2} [\ln( Z^{(2)}_{(i:n)})-\ln( Z^{(2)}_{(k_1:n)})]
\end{eqnarray*}
for some $k_2:=k_2(n)\in\{1,\ldots,n\}$.
Similarly, we use as estimator for $\alpha_0$, the index of hidden regular variation, the Hill-estimator based on the data $\min(Z_1^{(1)},Z_1^{(2)}),\ldots,\min(Z_n^{(1)},Z_n^{(2)})$.
Therefore, define $Z^{\min}_i=\min(Z_i^{(1)},Z_i^{(2)})$ for $i\in\N$. The order statistics of $Z^{\min}_1,\ldots,Z^{\min}_n$
are denoted by $Z^{\min}_{(1:n)}\geq\ldots\geq Z^{\min}_{(n:n)}$. The Hill-estimator for $\alpha_0$ is then
\begin{eqnarray*}
    \wh\alpha_{0,n}=\frac{1}{k_0}\sum_{i=1}^{k_0} [\ln (Z^{\min}_{(i:n)})-\ln( Z^{\min}_{(k_2:n)})]
\end{eqnarray*}
for some $k_0:=k_0(n)\in\{1,\ldots,n\}$.

\begin{Corollary} \label{Corollary:Hill}
Let  Assumptions \ref{cond:basic} and \ref{assumption D} hold. Furthermore, suppose
the following conditions are satisfied:
\begin{enumerate}
    \item $\min(k,k_0,k_2)\to\infty$, $\max(k,k_0,k_2)/n\to 0$ as $n\to\infty$.
    \item  $p_n\in(0,1)$ such that $p_n\downarrow 0$, $np_n=o(k)$ and $\ln( k /(np_n))=o(\min(\sqrt{k_0},\sqrt{k_2}))$ as $n\to\infty$.
    \item  
    The second order conditions
\begin{eqnarray*}
    \lim_{t\to\infty}\frac{\frac{b_0(tx)}{b_0(t)}-x^{1/\alpha_0}}{A_0(t)}=x^{1/\alpha_0}\frac{x^{\rho_0}-1}{\rho_0}
    \quad \text{ and } \quad
    \lim_{t\to\infty}\frac{\frac{b_2(tx)}{b_2(t)}-x^{1/\beta}}{A_2(t)}=x^{1/\beta}\frac{x^{\rho_2}-1}{\rho_2}, \quad x>0,
\end{eqnarray*}
 where $\rho_0,\rho_2\leq 0$ are constants and $A_0,A_2$ are positive or negative
functions hold.
    \item $\lim_{t\to\infty}A_0(t)=\lim_{t\to\infty}A_2(t)=0$.
    \item There exist finite constants $\lambda_0,\lambda_2$ such that
    \begin{eqnarray*}
        \lim_{n\to\infty}\sqrt{k_0}A_0\left(\frac{n}{k_0}\right)=\lambda_0
            \quad \mbox{ and }\quad
        \lim_{n\to\infty}\sqrt{k_2}A_2\left(\frac{n}{k_2}\right)=\lambda_2.
    \end{eqnarray*}
\end{enumerate}
Then \eqref{eq:6.8} is satisfied.
\end{Corollary}
\begin{proof}
From \cite[Theorem~3.2.5]{deHaanFerreirabook}  we know that $\sqrt{k_2}(\beta-\wh\beta_n)\stackrel{\mathcal{D}}{\to}\mathcal{N}$ as $n\to \infty$ where $\mathcal{N}$
is a normally distributed random variable. In particular, $\wh\beta_n\stackrel{P}{\to}\beta$ as $n\to \infty$. The analogous result
holds for $\wh\alpha_{0,n}$ as well. Since by assumption $\ln ({k}/{np_n})=o(\sqrt{k_i})$ ($i=0,2$), we obtain
\begin{eqnarray*}
    \sqrt{k_0}(\wh\alpha_{0,n}-\alpha_0)\frac{\ln \left(\frac{k}{np_n}\right)}{\sqrt{k_0}}\stackrel{P}{\to}0
    \quad \mbox{ and }\quad \sqrt{k_2}(\wh\beta_n-\beta)\frac{\ln \left(\frac{k}{np_n}\right)}{\sqrt{k_2}}\stackrel{P}{\to}0, \quad (n\to \infty)
\end{eqnarray*}
which is  condition \eqref{eq:6.8}.
\end{proof}

\begin{Remark}
In our simulation study  in Section \ref{subsec:simex} we take $k=k_1=k_2$. In the study of extreme values, the choice of $k$  plays an important role and much work goes on in this area; see \cite{scarrott:macdonald:2012} for a brief overview. We choose $k$ to be 10\% of $n$, which is an ad-hoc choice but often used in practice.
\end{Remark}

\begin{Remark}
An alternative to the Hill estimator is the probability weighted moment estimator
based on the block maxima method which is under some regularity condition consistent
and asymptotically normally distributed as presented  in \cite[Theorem 2.3]{Ferreira:deHaan}
and hence, satisfies \eqref{eq:6.8}. Moreover, the peaks-over-threshold (POT) method
is a further option to estimate $\alpha_0,\beta$ which satisfies as well under some
regularity conditions \eqref{eq:6.8}; for more details on the asymptotic behavior
of estimators based on the POT method see \cite{Smith:87}.
\end{Remark}

\subsection{Simulated Examples} \label{subsec:simex}

First we use our methods on a few simulated examples.

\begin{Example}[Gaussian copula]\label{simex:exGauss}

Suppose $(Z_1,Z_2)$  has identical Pareto marginal distributions with common parameter $\alpha>0$  and  a dependence structure given by a  Gaussian copula
$ C_{\Phi,\rho}(u,v)$ with $\rho\in(-1,1)$ as given in Example \ref{example:gausscop}. A further restriction from the same example  leads us 
to assume $\rho\in(1-\frac{2}{\alpha+1},1)$ so that $\lim_{p\to 0}\MME(p)=\infty$.

In the Gaussian copula model, we can numerically compute the value of $\MME(p)$ for any specific $0<p<1$.  In our study we generate the above  distribution for four sets of choices of parameters:
\begin{itemize}
\item[(a)] $\alpha=2, \rho =0.9 $. Hence $\alpha_0=2.1$.
\item[(b)]  $\alpha=2, \rho =0.5$. Hence $\alpha_0=2.67$.
\item[(c)] $\alpha=2.3, \rho =0.8$. Hence $\alpha_0=2.55$.
\item[(d)]  $\alpha=1.9, \rho =0.8$. Hence $\alpha_0=2.11$.
 \end{itemize}

  \begin{figure}[H]
  \begin{centering}
  \includegraphics[width=1.02\linewidth]{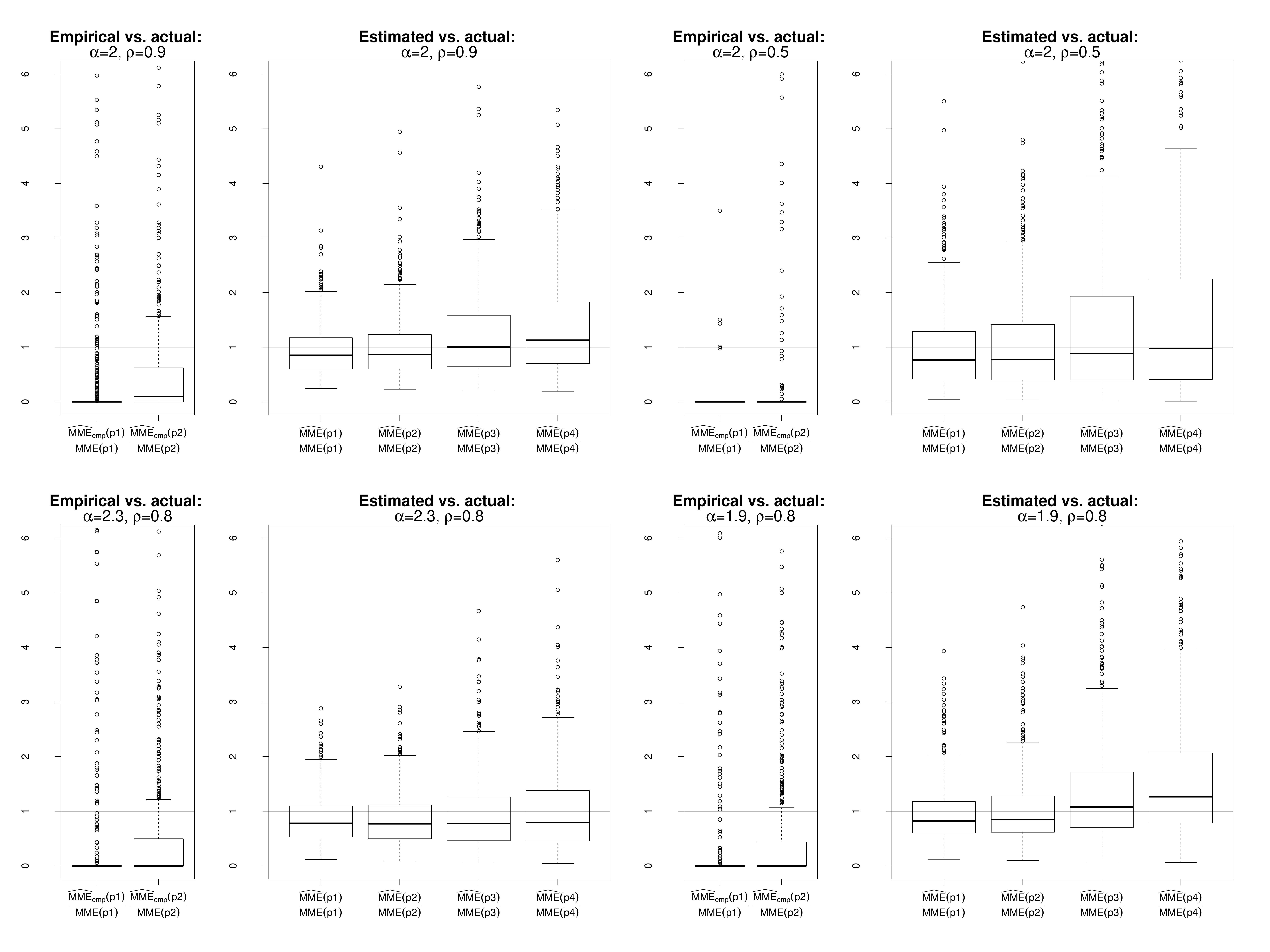}
  \end{centering}
  \caption{Box plots of  $\wh\MME_{\text{emp}}(p)/\MME(p)$ with $p1=1/500, p2=1/1000$ and  of $\wh\MME(p)/\MME(p)$ with $p1=1/500, p2=1/1000, p3=1/5000, p4=1/10000$ for Example \ref{simex:exGauss} with Gaussian copula: (a) top left: $\alpha=2, \rho=0.9$ and $\alpha_0=2.1$; (b) top right: $\alpha=2, \rho=0.5$ and $\alpha_0=2.67$; (c) bottom left: $\alpha=2.3, \rho=0.8$ and $\alpha_0=2.55$, (d) bottom right: $\alpha=1.9, \rho=0.8$ and $\alpha_0=2.11$.} \label{plot:gauss:box}  \end{figure}

The parameters $\alpha$ and $\alpha_{0}$ are estimated using the Hill estimator which appears to estimate the parameters quite well; see \cite{Resnick:2007} for details. The estimated values $\wh \alpha$ and $\alpha_{0}$ are used to
compute estimated values of MME.

In order to check the performance of the estimator when $p \ll 1/n$ we create box-plots for ${\widehat{\MME}}/{\MME}$ from 500 samples in each of the four models,
 where $n=1000,k=100$ and we restrict to 4 values of $p$ given by $1/500, 1/1000,1/5000,1/10000$. The plot is given in Figure \ref{plot:gauss:box}. Overall the ratio of the estimate to its real target value seem close to one, and we conclude that the estimators are reasonably good.

\end{Example}

\begin{Example}[Marshall-Olkin copula]\label{simex:exMarOlk}

Suppose $(Z_1,Z_2)$  has identical Pareto marginal distributions with parameter $\alpha>0$  and  a dependence structure given by a    \emph{Marshall-Olkin survival copula} with parameters $\gamma_1,\gamma_2\in(0,1)$ as given in Example \ref{example:marolcop}.

  \begin{figure}[h]
  \begin{centering}
  \includegraphics[width=1.03\linewidth]{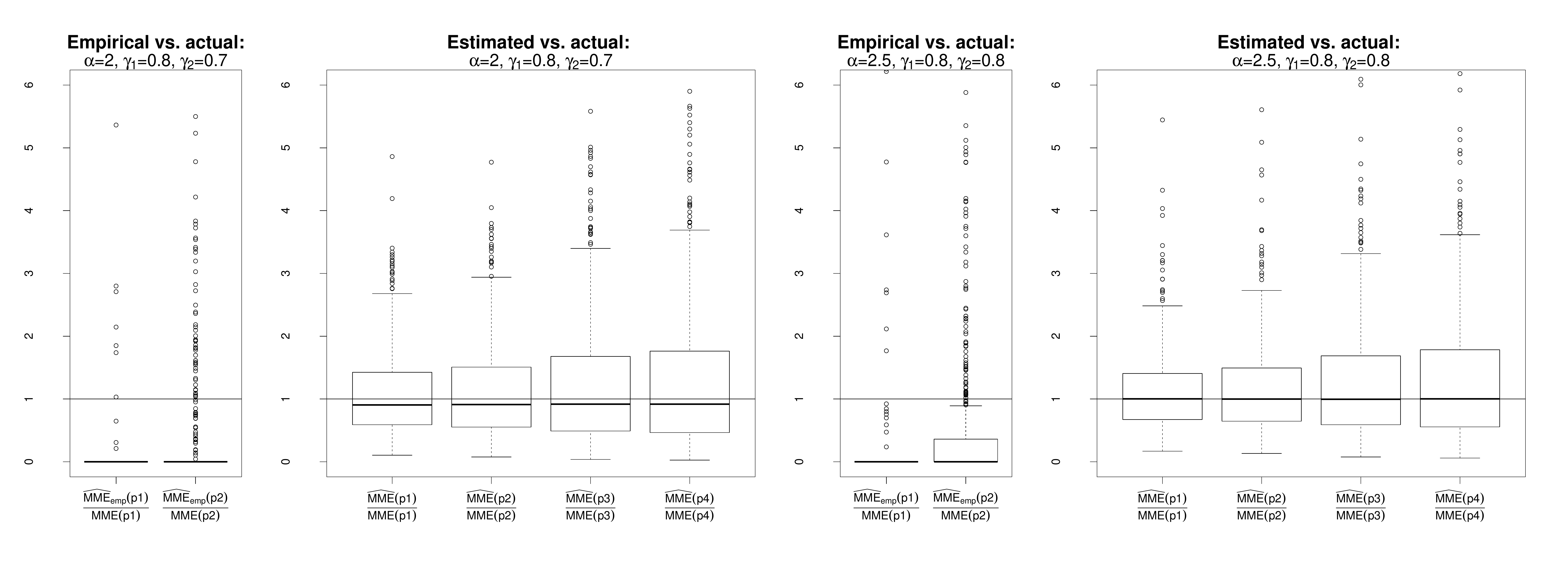}
  \end{centering}
  \caption{Box plots of  $\wh\MME_{\text{emp}}(p)/\MME(p)$ with $p1=1/500, p2= 1/1000$ and $\wh\MME(p)/\MME(p)$ with $p1=1/500, p2=1/1000, p3=1/5000, p4=1/10000$ for Example \ref{simex:exMarOlk} with Marshall-Olkin copula: (a) left two plots: $\alpha=2, \gamma_1=0.8,\gamma_2=0.7$ and $\alpha_0=2.6$; {(b) }right two plots: $\alpha=2.5, \gamma_1=0.8,\gamma_2=0.8$ and $\alpha_0=3$.}  \label{plot:MO:box} \end{figure}
We note that a parameter restriction from Example \ref{example:marolcop} is given by $\min(\gamma_1,\gamma_2)\in(1-1/\alpha,1)$.  Hence, we find estimates of $\MME$ for the $\gamma_1\ge\gamma_2$ case but not for $\MES$ in this example.
 For  $\gamma_1\ge\gamma_2$, we can explicitly compute
\begin{align*}
\MME(p)  = \frac1{\alpha-1} {p^{1-\gamma_2-1/{\alpha} }}.  \end{align*}

In our study we generate the above  distribution for two sets of choice of parameters:
\begin{itemize}
\item[(a)] $\alpha=2, \gamma_1=0.8, \gamma_2=0.7$. Hence $\alpha_0=2.6$.
\item[(b)] $\alpha=2.5, \gamma_1=0.8, \gamma_2=0.8$. Hence $\alpha_0=3$.
\end{itemize}

In  \Cref{plot:MO:box}, we create box-plots for ${\widehat{\MME}}/{\MME}$ from 500 samples in each of the four models,
 where $n=1000,k=100$ and we restrict to 4 values of $p$ given by $1/500, 1/1000,1/5000,1/10000$.  Again we observe that the ratio of the estimate to its real target value seem to be close to one, and we conclude that the estimators are reasonably good.

\end{Example}

\begin{figure}[H]
  \includegraphics[width=1.01\linewidth]{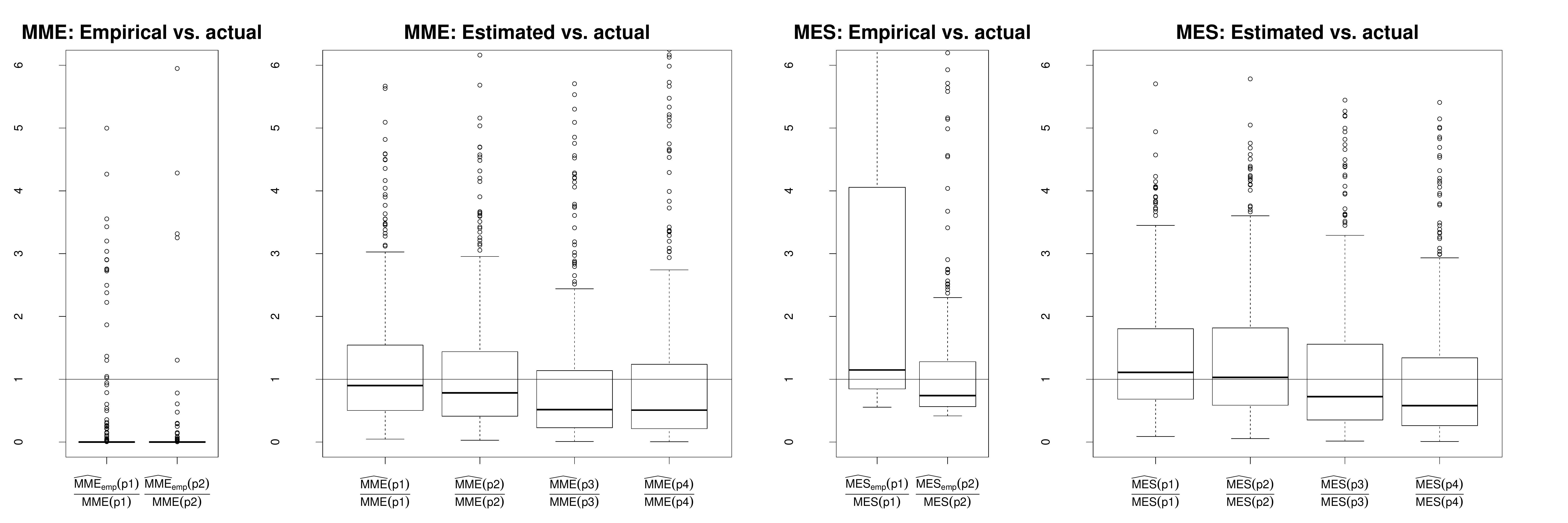}
  \caption{(a) Left two plots: Box plots of  $\wh\MME_{\text{emp}}(p)/\MME(p)$  with $p1=1/500, p2=1/1000$ as well as $\wh\MME(p)/\MME(p)$  with $p1=1/500, p2=1/1000, p3=1/5000, p4=1/10000$ for Model C in \Cref{simex:exModelC} with $\alpha=1.5, \alpha_0=2$. (b) Right two plots: Analog plots for $\MES$.}  \label{plot:modelC:box}  \end{figure}

\begin{Example}[Model C]\label{simex:exModelC}

We look at \Cref{additive:model} where $ \bY=(Y_1,Y_2)$ and $Y_1,Y_2$ are iid Pareto $(\alpha)$ random variables, $\bV=(V_1,V_2)$ with $V_1=V_2$ following Pareto $(\alpha_0)$ and $\bZ=\bY+\bV$.  Using \Cref{theorem:5.3} we can check that $\bZ\in \MRV(\alpha,b,\nu) \cap \HRV(\alpha_0,b_0,\nu_0)$ if $\alpha<\alpha_0<\alpha+1$ and all conditions (A), (B1) and (B2) are satisfied. Thus, we can find limits for both $\MME(p)$ and $\MES(p)$ for $p$ going to 0. It is also possible to calculate MME and MES explicitly. We do so for $\alpha=1.5$ and $\alpha_0=2$ here.

We found that the Hill plots were not that stable, hence we used an L-moment estimator (a probability weighted moment estimator could be used as well) to estimate $\alpha$ and $\alpha_0$; see \cite{hosking:1990,deHaanFerreirabook} for details.
The estimates of the tail parameters are not shown here. In Figure \ref{plot:modelC:box}, we create box-plots for ${\widehat{\MME}}/{\MME}$ and ${\widehat{\MES}}/{\MES}$
 where $n=1000,k=100$ with 500 samples and we restrict to 4 values of $p$ given by $1/500, 1/1000,1/5000,1/10000$. The ratios of the estimators and the targets seem close to one. Of course, the empirical estimators for $p=1/500, p=1/1000$ do not perform so well.

\end{Example}

\subsection{Data Example: Returns from Netflix and S\&P}  \label{subsec:realex}

In this section we use the method we developed in order to estimate MME and MES from  a real data set. We observe a data set which exhibits \emph{asymptotic tail independence} and we compare estimates of both statistics (MME and MES) under this assumption versus a case when we use a formula that does not assume asymptotic independence (similar to estimates obtained in  \cite{cai:einmahl:dehaan:zhou:2015}).

 \begin{figure}[h]
  \begin{centering}
  \includegraphics[width=0.80\linewidth]{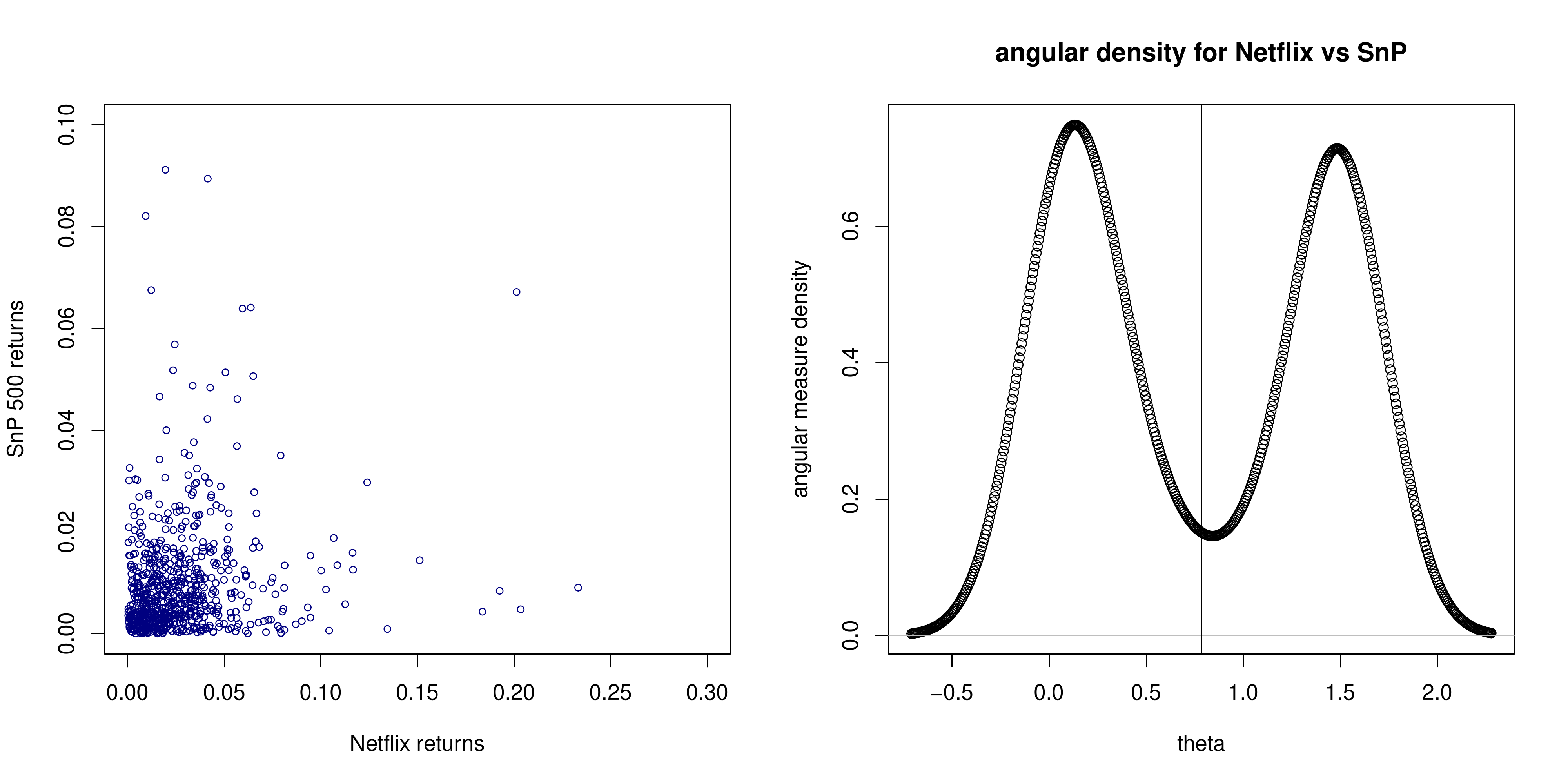}
  \end{centering}
  \caption{Left plot: Scatter plot of (NFLX, SNP). Right plot: angular density plot to of the rank-transformed returns data. }  \label{plot:Netflix:A}
  \end{figure}

We observe return values from daily equity prices of Netflix (NASDAQ:NFLX) as well as daily return values from S\&P 500 index for the period January 1, 2004 to December 31, 2013. The data was downloaded from \emph{Yahoo Finance} (\url{http://finance.yahoo.com/}). The entire data set uses 2517 trading days out of which 687 days exhibited negative returns in both components and we used these 687 data points for our study.

  \begin{figure}[h]
  \begin{centering}
  \includegraphics[width=\linewidth]{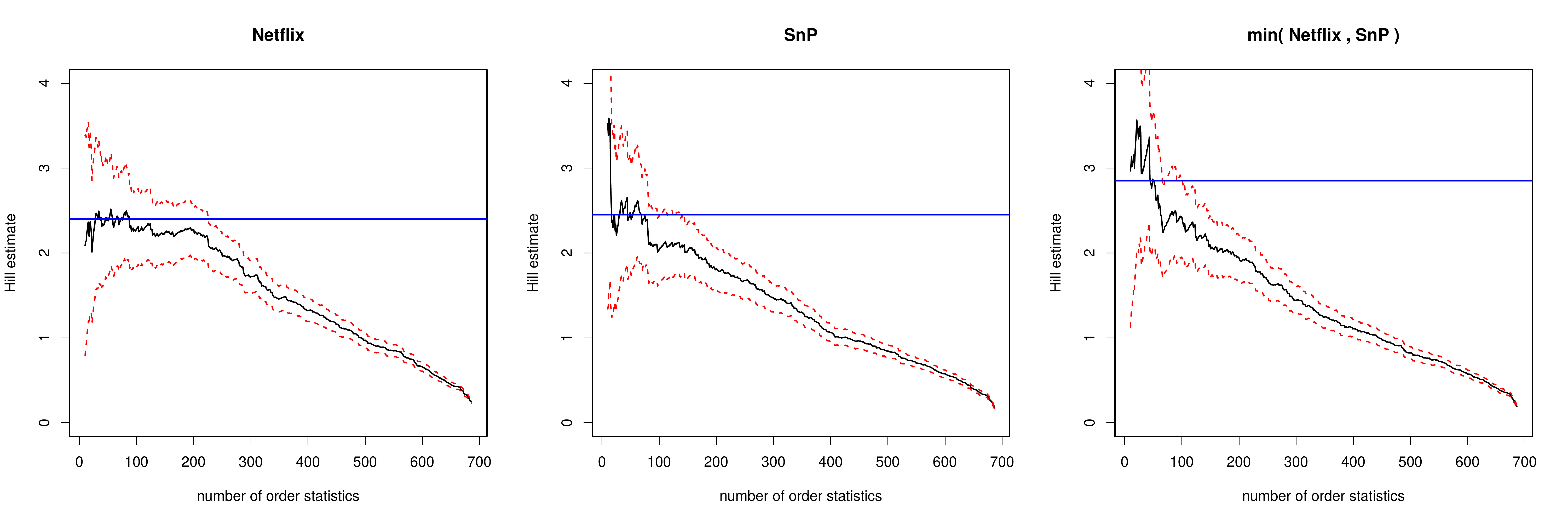}
  \end{centering}
  \caption{Hill plots of the tail parameters of the two negative returns (NFLX,SNP) and that of hidden tail parameter $\alpha_{0}$ estimated using $\min$(NFLX,SNP).}  \label{plot:Netflix:B}
  \end{figure}

A scatter plot of the returns data shows some concentration around the axes but the data seems to exhibit some positive dependence of the variables too; see leftmost plot in \Cref{plot:Netflix:A}.  Since the scatterplot doesn't clearly show whether the data has asymptotic tail independence or not, we create an angular density plot of the rank-transformed data. Under asymptotic independence we should observe two peaks in the density, one concentrating around 0 and the other around $\pi/2$, which is what we see in the right plot in \Cref{plot:Netflix:A}; see \cite{Resnick:2007} for further discussion on the angular density.  Hence, we can discern that our data exhibits asymptotic tail independence and proceed to compute the hidden regular variation tail parameter using $\min$(NFLX, SNP) as the data used to get a Hill estimate of $\alpha_{0}$.
The left two plots in \Cref{plot:Netflix:B} show Hill plots of both the Netflix negative returns  (NFLX) and the S\&P 500 negative returns (SNP). A QQ plot (not shown) suggests that both margins are heavy-tailed and by choosing $k=50$ for the Hill-estimator we obtain as estimate of the tail parameters $\wh \alpha_{\text{NFLX}}=2.39, \wh \alpha_{\text{SNP}}=2.46$ (indicated by blue horizontal lines in the plot). Again using a Hill-estimator with $k=50$, the estimate $\wh \alpha_{0} =2.86$ is obtained; see the rightmost plot in  \Cref{plot:Netflix:B}.

Now, we use the values of  $\wh \alpha_{\text{SNP}}=2.46$ and $\wh \alpha_{0} =2.86$ to compute estimated values of MME and MES. In \Cref{plot:Netflix:C} we plot the empirical estimates of MME and MES (dotted lines), the extreme value estimate without assuming asymptotic independence (blue bold line) and the extreme value estimate assuming asymptotic independence (black bold line). We observe that both MME and MES values are smaller under the assumption of asymptotic independence than in the case where we do not assume asymptotic independence.
Hence, without an assumption of asymptotic independence, the firm might over-estimate its' capital shortfall if the systemic returns tend to show an extreme loss.
  \begin{figure}[h]
  \begin{centering}
  \includegraphics[width=0.8\linewidth]{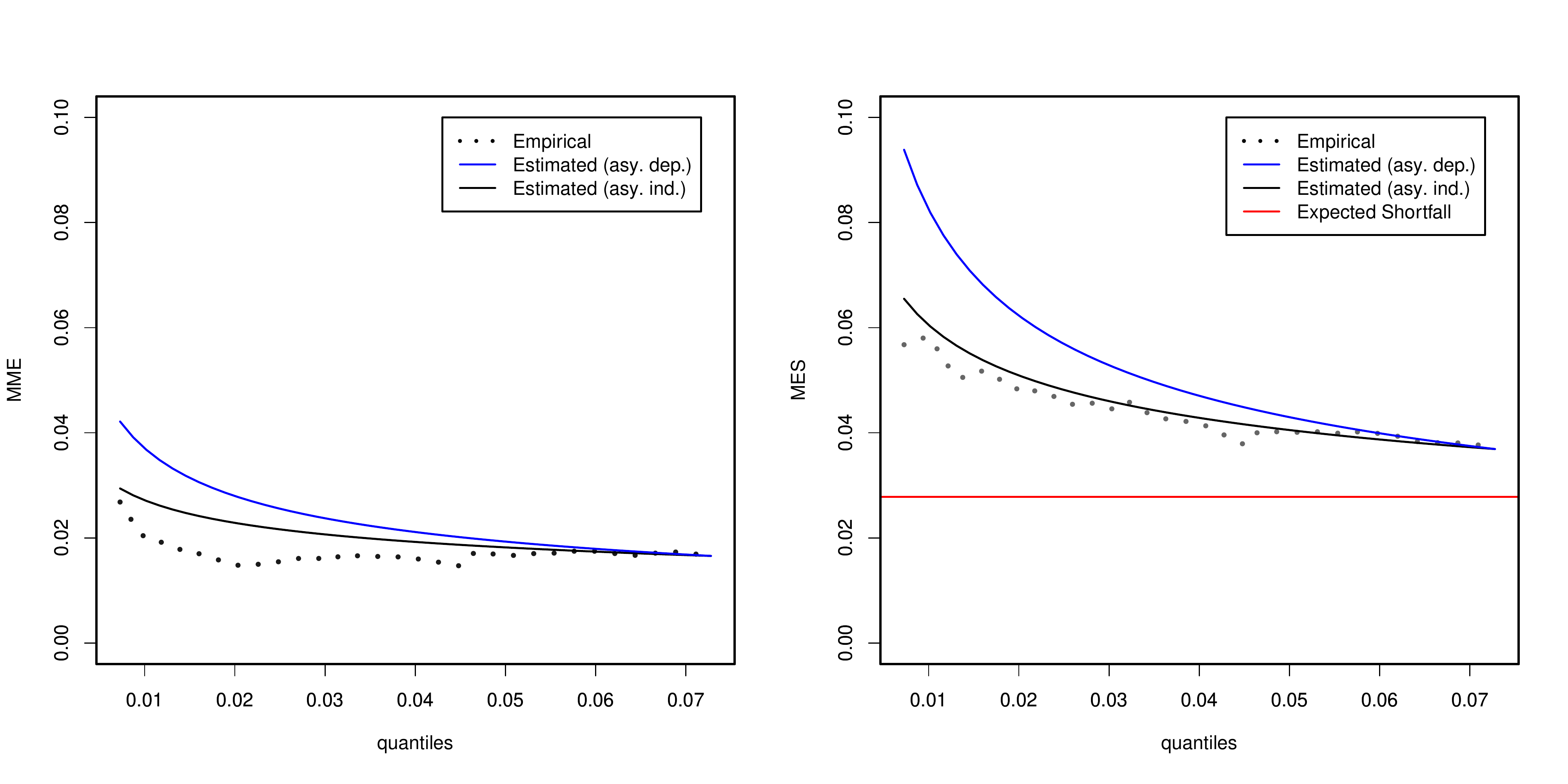}
  \end{centering}
  \caption{MME and MES plots under the tail dependence model as well as the asymptotic independent model.}  \label{plot:Netflix:C}
  \end{figure}

\section{Conclusion} \label{sec:conclusion}
 In this paper we study two measures of systemic risk, namely \emph{Marginal Expected Shortfall} and \emph{Marginal Mean Excess} in the presence of asymptotic independence of the marginal distributions in a bivariate set-up. We specifically observe that the very useful Gaussian copula model with Pareto-type tails satisfies our model assumptions for the MME and we can find the right rate of increase (decrease) of MME in this case. Moreover we observe that if the data exhibit \emph{asymptotic tail independence}, then we can provide an  estimate of MME that is closer to the empirical estimate (and possibly smaller) than the one that would be obtained if we did not assume \emph{asymptotic tail independence}.

 In a companion paper, \cite{Fasen:Das:2017}, we investigate various copula models and mixture models which satisfy our assumptions under which we can find asymptotic limits of MME and MES. A further direction of work would involve finding the influence of multiple system-wide risk events (for example, multiple market indicators) on a single or group of components (for example, one or more financial institutions).

\bibliographystyle{plainnat}
\bibliography{bibeshrv}

\end{document}